\documentclass[letterpaper,11pt]{amsart}
\usepackage{etex}
\usepackage[all]{xy}
\usepackage{amssymb,amscd,amsmath,amsfonts,amsthm}
\usepackage{stmaryrd,mathabx,tikz,mathrsfs,float,enumitem}
\usepackage{ulem}
\usepackage{graphpap,color}
\usepackage{pstricks}
\usepackage{cancel}
\usepackage[verbose]{hyperref}

\numberwithin{equation}{section}

\def\wh#1{\widehat{#1}}
\def\wc#1{\widecheck{#1}}
\def\ti#1{\widetilde{#1}}
\def\ov#1{\overline{#1}}
\def\ud#1{\underline{#1}}

\def\tn{\textnormal}
\def\ts{\textsf}

\def\inn{\!\in\!}
\def\eq{\!=\!}
\def\bsl{\backslash}
\def\lra{\longrightarrow}
\def\prt{\partial}

\def\lr#1{\langle{#1}\rangle}

\def\lbrp#1{\llbracket{#1}\rrparenthesis}
\def\lprp#1{\llparenthesis{#1}\rrparenthesis}

\def\mwt{\fM^{\tn{wt}}}
\def\mtd#1{\fM^{\tn{tf}}_{#1}}
\def\tmwt{\ti\fM^{\tn{wt}}}

\def\wt{\tn{wt}}

\def\Edg{\tn{Edg}}

\def\Ver{\tn{Ver}}
\def\Edg{\tn{Edg}}

\def\A{\mathbb A}

\def\cC{\mathcal C}

\def\sE{\mathscr E}
\def\fE{\mathfrak E}

\def\bfE{\mathbf E}

\def\bbI{\mathbb I}
\def\fL{\mathfrak L}

\def\nL{\mathsf L}
\def\fM{\mathfrak M}

\def\fN{\mathfrak N}
\def\sN{\mathscr N}
\def\sO{\mathscr O}

\def\P{\mathbb P}
\def\cP{\mathcal P}

\def\nR{\mathsf R}
\def\R{\mathbb R}
\def\fS{\mathfrak S}
\def\sT{\mathscr T}
\def\cU{\mathcal U}
\def\fU{\mathfrak U}
\def\cV{\mathcal V}

\def\Z{\mathbb Z}
\def\cZ{\mathcal Z}

\def\sa{\mathsf a}

\def\sd{\mathsf d}
\def\se{\mathsf e}
\def\fe{\mathfrak e}
\def\ff{\mathfrak f}

\def\bu{\mathbf u}
\def\sv{\mathsf v}
\def\sw{\mathsf w}
\def\bfw{\mathbf w}
\def\fz{\mathfrak z}

\def\Ga{\Gamma}

\def\al{\alpha}

\def\ga{\gamma}
\def\de{\delta}

\def\ka{\kappa}
\def\la{\lambda}

\def\si{\sigma}
\def\th{\theta}
\def\ve{\varepsilon}

\def\vs{\varsigma}

\def\ze{\zeta}

\def\cht{{\mathbf{m}}}

\def\tr{{\tau}}
\def\lt{{\mathfrak t}}

\def\td{\tn{tf}}
\def\tf{\tn{tf}}
\def\wt{\tn{wt}}

\usetikzlibrary{arrows}
\usetikzlibrary{decorations.pathreplacing}
\usetikzlibrary{decorations.pathmorphing}
\usetikzlibrary{patterns}

\newtheorem{thm}{Theorem}[section]
\newtheorem{prp}[thm]{Proposition} 
\newtheorem{lmm}[thm]{Lemma}  
\newtheorem{crl}[thm]{Corollary}

\newtheorem{dfn}[thm]{Definition}

\theoremstyle{remark}
\newtheorem{rmk}[thm]{Remark}

\title
{Moduli of  Curves of Genus One with Twisted Fields}
\date{}
\author{Yi Hu}
\address{Department of Mathematics, University of Arizona, USA.}
\email{yhu@math.arizona.edu}
\author{Jingchen Niu}
\address{Department of Mathematics, University of Arizona, USA.}
\email{jniu@math.arizona.edu}

\begin{document}

\begin{abstract} 
 We construct a smooth Artin stack
 parameterizing the stable weighted curves of genus one with twisted fields and  prove that it is isomorphic to the blowup stack of the moduli of genus one weighted curves studied by Hu and Li.
 This leads to a blowup-free construction of Vakil-Zinger's desingularization of the moduli of genus one stable maps to projective spaces. 
This construction provides the cornerstone of the theory of stacks with twisted fields,
which is thoroughly studied in~\cite{HN2} and leads to a blowup-free resolution of the stable map moduli of genus two.
\end{abstract}

\maketitle 
\section{Introduction}\label{Sec:Intro}
Moduli problems are of central importance in algebraic geometry. 
Many moduli spaces possess arbitrary singularities~\cite{V}.
Among them, the moduli $\ov M_g(\P^n,d)$ of degree $d$ stable maps from genus $g$ nodal curves into  projective spaces $\P^n$ 
are particularly important. 
We aim to resolve the singularities of $\ov M_g(\P^n,d)$,
that is, to construct a new Deligne-Mumford stack that has smooth irreducible components and normal crossing boundaries and dominates $\ov\fM_g(\P^n,d)$  properly and birationally onto the primary component (the component whose general points have smooth domain curves).
The problem of resolution of singularities is arguably among the hardest ones 
in algebraic geometry \cite{Hironaka64a, Hironaka64b, deJong96, K07}.  

The stable map moduli are smooth  if $g\eq 0$
and singular if $g\!\ge\!1$ and $n\!\ge\!2$.
For $g\eq 1$,  a resolution was constructed by Vakil and Zinger~\cite{VZ08}, 
followed by an algebraic approach of Hu and Li~\cite{HL10}. 
The latter is achieved by constructing a canonical smooth blowup 
$\ti\fM^\wt_1$ of the Artin stack 
$\mwt_1$ of weighted nodal curves of genus one.
The method of \cite{HL10} was further developed  in~\cite{HLN} to finally  establish
 a resolution in the case of $g\eq 2$.
 The resolution of~\cite{HLN} is achieved by constructing a canonical smooth blowup 
$\widetilde{\mathfrak P}_2$ of the relative Picard stack 
${\mathfrak P}_2$ of  nodal curves of genus two.

 In higher genus cases, the construction of a possible resolution of the stable map moduli 
may seem formidable. The constructions of the explicit resolutions in~\cite{VZ08,HL10,HLN}  rely on certain precise knowledge on the singularities of the moduli.
For arbitrary genus, it calls for a more abstract and geometric approach. 
As advocated by the first author,
every singular moduli space should admit a resolution which itself is also a moduli. 
Following this principle, 
we interpret the blowup stack $\ti\fM_1^\wt$ of~\cite{HL10} as a smooth
algebraic stack of stable weighted nodal curves of genus one with twisted fields,
and consequently, the  resolution $\ti M_1(\P^n,d)$ of $\ov M_1(\P^n,d)$
as a Deligne-Mumford stack of genus one stable maps with twisted fields. 
The results in this paper are the first step to tackle the arbitrary genus case.

The main theorem of this paper is the following:

\begin{thm}
\label{Thm:Main} There exits a smooth Artin stack $\mtd{1}$
 parameterizing the weighted nodal curves of genus one with twisted fields,
along with a universal family $\cC^\tf\lra\mtd 1$ 
and a proper and birational forgetful morphism $\varpi: \mtd 1\lra\mwt_1$.
Moreover, $\mtd{1}/ \mwt_1$ is isomorphic to
the blowup stack $\ti\fM_1^\wt/\mwt_1$.
\end{thm}

 We construct the strata of $\mtd 1$ and the forgetful map $\varpi$ in \S\ref{Sec:Set-theoretic}; see~(\ref{Eqn:MtdStrata}). 
We then glue the strata of $\mtd 1$ together using smooth charts in \S\ref{Sec:Moduli} and 
conclude that $\mtd 1$ is a smooth Artin stack and is birational to $\mwt_1$ in Corollary~\ref{Crl:MtdSmooth}.
The universal family $\cC^\tf\!\lra\!\mtd 1$ is described in Proposition~\ref{Prp:Stack}.
We finally show that $\mtd 1/\mwt_1$ is  isomorphic to $\ti\fM_1^\wt/\mwt_1$ in Proposition~\ref{Prp:Isomorphism},
which implies the properness of $\varpi$.
These results together establish Theorem~\ref{Thm:Main}.
 
We remark that a direct approach to the properness of $\varpi$ (i.e.~without the comparison with the blowup $\ti\fM^\wt_1/\mwt_1$) is provided in the proof of~\cite[Theorem~2.19($\mathsf{p_1}$)]{HN2},
in a more general setting.

We also point out that there should exist a groupoid, represented by  $\mtd{1}$, that sends any scheme~$S$ to the set of the flat families of stable weighted nodal curves of genus 1 with twisted fields over $S$ as in~(\ref{Eqn:family});
see Remark~\ref{Rmk:moduli} for some details.

According to~\cite{HL10},
the resolution $\ti M_1(\P^n,d)$ of $\ov M_1(\P^n,d)$ is given by
$$\ti M_1(\P^n,d)= \ov M_1(\P^n,d)\times_{\mwt_1}\ti\fM_1^\wt,$$
where 
$$\ov M_1(\P^n,d)\lra\mwt_1,\qquad
[C,\bu]\mapsto[C,c_1(\bu^*\sO_{\P^n}(1))]$$
and $\ti\fM^\wt_1\!\lra\!\mwt_1$ is the canonical blowup.
Analogously,
we take
$$
\ti M_1^\tf(\P^n,d):=
\ov M_1(\P^n,d)\times_{\mwt_1}\mtd 1,
$$
where $\mtd 1\!\lra\!\mwt_1$ is the forgetful morphism aforementioned.
Theorem~\ref{Thm:Main} then leads to the following conclusion immediately.

\begin{crl}
\label{Crl:Main}
$\ti M_1^\tf(\P^n,d)$ is a proper Deligne-Mumford stack and is isomorphic to
$\ti M_1(\P^n,d)$.
\end{crl}

Via the above isomorphism and applying ~\cite{HL10}, 
one sees that the stack $\ti M_1^\tf(\P^n,d)$ provides a resolution of 
$\ov M_1(\P^n,d)$.
Nonetheless, without relating to  $\ti M_1(\P^n,d)$,
we can directly prove the resolution property of $\ti M_1^\tf(\P^n,d)$ by
investigating the local equations of $\ov M_1(\P^n,d)$ in~\cite{HL10} 
and their pullbacks to $\ti M_1^\tf(\P^n,d)$;
see Remark~\ref{Rmk:smooth}.
 
The methods and ideas of this paper are essential to the development in \cite{HN2} and forthcoming works.
Based on the construction of $\mtd 1$,
we introduce the theory of \ts{stacks with twisted fields} (\ts{STF}) in~\cite[Theorem~2.19]{HN2}.
To be somewhat more informative,
we work on a smooth stack $\fM$ that has a stratification indexed by a set $\Ga$ of graphs similar to~(\ref{Eqn:Mwt_strata}); see~\cite[Definition~2.15]{HN2}.
The graphs in $\Ga$ need not to come from the dual graphs as in~(\ref{Eqn:Mwt_strata}), 
but the stratification of $\fM$ should resemble~(\ref{Eqn:V_(I)}) locally.
Moreover, $\Ga$ need not to consist of trees,
but it should contain necessary information on the notion of the (weighted) level trees in Definition~\ref{Dfn:level_tree} so that we can add the twisted fields to the strata of $\fM$ parallel to~(\ref{Eqn:MtdStrat_and_sEcT}) and obtain a new stack $\mtd{}$; see~\cite[Definition~2.17]{HN2}.
Such $\mtd{}$ enjoys desirable properties as in Corollary~\ref{Crl:MtdSmooth} and Remark~\ref{Rmk:smooth}.
As an application of the STF theory,
in~\cite{HN2},
we construct a smooth Artin stack~$\mathfrak P_2^\tf$ of genus~2 nodal curves with line bundles and twisted fields, along with a proper and birational forgetful morphism $\mathfrak P_2^\tf\lra\mathfrak P_2$,
such that 
$$
 \ti M_2^\tf(\P^n,d)=\ov M_2(\P^n,d)\times_{{\mathfrak P}_2}\mathfrak P_2^\tf
   \lra \ov M_2(\P^n,d)
$$
provides a resolution.
Further, we expect that they can be extended to arbitrary genus,  as far as the existence of 
 moduli of nodal curves with twisted fields is concerned.
This is the main motivation of the current article.
 
In a related work ~\cite{RSPW},  D. Ranganathan, K. Santos-Parker, and J.Wise
provide a different modular perspective of  $\ti M_1(\P^n,d)$ using logarithmic geometry.

\textbf{Acknowledgments.} We would like to thank Dawei Chen, Qile Chen, and Jack Hall for the valuable discussions. 

\textbf{Convention.}
The subscript ``1'' of the relevant stacks indicating the genus appears only in \S\ref{Sec:Intro} and will be omitted starting \S\ref{Sec:Set-theoretic},
as we only deal with the genus 1 case in this paper.
In particular,
we will denote by
$$
\mwt\qquad\tn{and}\qquad
\mtd{}
$$
the aforementioned stacks $\mwt_1$ and $\mtd 1$, respectively.

\section{Set-theoretic descriptions}
\label{Sec:Set-theoretic}

In \S\ref{Subsec:Level Tree},
we discuss the combinatorics of the dual graphs of nodal curves and introduce the notion of the weighted level trees.
They will be used to define $\mtd{}$ set-theoretically in \S\ref{Subsec:twisted_fields}.

\subsection{Weighted level trees}
\label{Subsec:Level Tree}

Let
$\ga$ be a \ts{rooted tree},
i.e.~a connected finite graph that contains no cycles, along with a special vertex $o$, called the \ts{root}. 
The sets of the vertices and the edges of $\ga$ are denoted by
$$
 \Ver(\ga)\qquad\tn{and}\qquad
 \Edg(\ga),
$$
respectively. The set
$\Ver(\ga)$ is endowed with a partial order, called the \ts{tree order},
so that 
$v\!\succ\!v'$ if and only if $v\!\ne\!v'$ and $v$ belongs to a path between $o$ and~$v'$.
The root $o$ is thus the unique maximal element of $\Ver(\ga)$ with respect to the tree order.

For each $e\inn\Edg(\ga)$,
we denote by $v_e^\pm\inn\Ver(\ga)$ the endpoints of $e$ such that
$
 v_e^+\!\succ\!v_e^-.
$
Then, every vertex $v\inn\Ver(\ga)\bsl\{o\}$ corresponds to a unique 
$$
 e_v\in\Edg(\ga)\qquad
 \tn{satisfying}\qquad
 v_{e_v}^-=v.
$$
The tree order on $\Ver(\ga)$ induces a partial order on $\Edg(\ga)$, still called the \ts{tree order},
so that
$$
 e\succ e'\quad\Longleftrightarrow\quad
 v_e^-\succeq v_{e'}^+\,.
$$

We call a pair $\tr\eq(\ga,\bfw)$ consisting of a rooted tree $\ga$ and a function $$\bfw:\Ver(\ga)\lra\Z_{\ge 0}$$ a~\ts{weighted tree}.
For such $\tr$,
we write
$\Ver(\tr)\eq\Ver(\ga)$ and $\Edg(\tr)\eq\Edg(\ga)$.
The set of all the weighted trees is denoted by $\sT_\nR^\wt$.

We call a map $\ell\!:\Ver(\ga)\lra\mathbb R_{\le 0}$ satisfying
$$
 \ell^{-1}(0)\eq\{o\}
 \qquad\tn{and}
 \qquad
 \ell(v)\!>\!\ell(v')\ \ \tn{whenever}\ \ v\!\succ\!v'
$$
a \ts{level map}.
For each $i\inn\ell(\Ver(\ga))\bsl\{0\}$,
let 
\begin{equation}\label{Eqn:i^sharp}
i^\sharp=\min\big\{\,
k\inn\ell\big(\Ver(\ga)\big):\,
k\!>\!i\,\big\},
\end{equation}
i.e.~the level $i^\sharp$ is right ``above'' the level $i$;
see Figure~\ref{Fig:level_tree}.
We remark that a rooted tree along with a level map is called a {\it level graph} with the root as the unique {\it top level} vertex in~\cite[\S1.5]{BCGGM}.

\begin{dfn}\label{Dfn:level_tree}
We call the tuple
\begin{equation*}
 \lt=\big(\,\ga,\;
 \bfw\!:\Ver(\ga)\!\lra\!\Z_{\ge 0}\,,\;
 \ell\!:\Ver(\ga)\!\lra\!\mathbb R_{\le 0}\,\big)
\end{equation*}
a \ts{weighted level tree} if $(\ga,\bfw)\inn\sT_\nR^\wt$ and $\ell$ is a level map.
\end{dfn}

For every weighted level tree $\lt$ as above,
we write $\Ver(\lt)\eq\Ver(\ga)$ and $\Edg(\lt)\eq\Edg(\ga)$.
Set
\begin{align*}
 &\cht=\cht(\lt)=
 \max\big\{\ell(v):
 v\inn\Ver(\lt), \bfw(v)\!>\!0\big\}
 &&(\,\le 0\,),\\
 &\wh\Edg(\lt)=
 \big\{e\inn\Edg(\lt):
 \ell(v_e^+)\!>\!\cht\big\}
 && \big(\,\subset\,\Edg(\lt)\big).
\end{align*}
For any two levels $i,j\inn\R_{\le 0}$,
we write
\begin{equation}\label{Eqn:lrbr}
\lprp{i,j}_\lt=\ell\big(\Ver(\lt)\big)\!\cap\!(i,j),\quad
\lbrp{i,j}_\lt=
\ell\big(\Ver(\lt)\big)\!\cap\![i,j).
\end{equation}
For every $e\inn\wh\Edg(\lt)$,
let
$$
 \ell(e)=\max\{\ell(v_e^-),\cht\}
 \quad\big(\in\lbrp{\cht,0}_\lt\big).
$$
For each level $i\inn\lbrp{\cht,0}_\lt$,
we set
\begin{equation}
\label{Eqn:fE_i}
\fE_i=\fE_i(\lt)=\big\{\,e\inn\wh\Edg(\lt):
 \ell(e)\!\le\!i\!<\!\ell(v_e^+)\,\big\}.
\end{equation}
In other words,
$\fE_i$ consists of all the edges crossing the gap between the levels~$i$ and $i^\sharp$.

We remark that all the notions in the preceding paragraph depend on the weighted level tree $\lt$,
although 
we may hereafter omit $\lt$ in any of such notions when the context is clear. 

Every weighted level tree $\lt$ determines a unique index set
\begin{equation}\begin{split}\label{Eqn:bbI}
 &\bbI(\lt) =
 \bbI_+(\lt)\sqcup\bbI_\cht(\lt)\sqcup\bbI_-(\lt),
 \hspace{.59in}\tn{where}\quad
 \bbI_+(\lt)\eq\lbrp{\cht,0}_\lt,\\
 &\bbI_\cht(\lt)\eq \{e\inn\wh\Edg(\lt)\!:\ell(v_e^-)\!<\!\cht\},\qquad
 \bbI_-(\lt)\eq
 \big(\Edg(\lt)\bsl\wh\Edg(\lt)\big).
\end{split}\end{equation}
The set $\bbI_+(\lt)$ becomes empty if $\cht\eq 0$,
i.e.~the root $o$ is positively weighted.
As mentioned before,
we may simply write
$$
 \bbI=\bbI(\lt),\quad
 \bbI_\pm=\bbI_\pm(\lt),\quad
 \bbI_\cht=\bbI_\cht(\lt)
$$
when the context is clear.

For each $I\!\subset\!\bbI$ (possibly empty), let
$$
 I_\cht=I\cap\bbI_\cht,\qquad
 I_\pm=I\cap\bbI_\pm.
$$
We construct a weighted level tree 
\begin{equation}\label{Eqn:t(I)}
 \lt_{(I)}=
 \big(\tau_{(I)},\ell_{(I)}\big)=
 \big(\,\ga_{(I)},\,\bfw_{(I)},\,
 \ell_{(I)}\,\big)
\end{equation}
as follows:
\begin{itemize}
[leftmargin=*]
\item 
the rooted tree $\ga_{(I)}$ is obtained via the edge contraction
$$
 \pi_{(I)}: \Ver(\ga)\twoheadrightarrow\Ver\big(\ga_{(I)}\big)
$$
such that the set of the contracted edges is
\begin{equation}\label{Eqn:edges_cntrd}
 \Edg(\lt)\bsl\Edg(\lt_{(I)})\eq
 \big\{e\inn\big(\wh\Edg(\lt)\bsl\bbI_\cht\big)\!\sqcup\! I_\cht\!:
 \lbrp{\ell(e),\ell(v_e^+)}_\lt\!\subset\!I_+\big\}\!\sqcup\!I_-;
\end{equation}
\item 
the weight function $\bfw_{(I)}$ is given by 
$$
 \bfw_{(I)}\!:\Ver\big(\ga_{(I)}\big)\lra\Z_{\ge 0},\qquad
 \bfw_{(I)}(v)=\!\!\sum_{v'\in\pi_{(I)}^{-1}(v)}
 \!\!\!\!\!\bfw(v');
$$

\item
the level map $\ell_{(I)}$ is such that for any $e\inn\Edg\big(\ga_{(I)}\big)~\big(\subset\!\Edg(\lt)\big)$,
$$
\ell_{(I)}(v_e^-)=
\begin{cases}
 \min\{i\inn\bbI_+\!\bsl I_+\!:\,
 i\!\ge\!\ell(v)~\forall\,v\inn\pi^{-1}_{(I)}(v_e^-)\} &
 \tn{if}~e\inn\wh\Edg(\lt)\!\bsl\bbI_\cht,
 \\
 \min(\bbI_+\!\bsl I_+) &
 \tn{if}~e\inn I_\cht,
 \\
 \max\{\ell(v)\!:\,
 v\inn\pi^{-1}_{(I)}(v_e^-)\} &
 \tn{if}~e\inn (\bbI_\cht\!\bsl I_\cht)\!\sqcup\!\bbI_-.
\end{cases}
$$
\end{itemize}
It is a direct check that $\tr_{(I)}$ is a weighted tree and $\ell_{(I)}$ satisfies the criteria of a level map,
hence~(\ref{Eqn:t(I)}) gives a well defined weighted level tree.

The construction of $\lt_{(I)}$ implies
\begin{equation}\label{Eqn:t_(I)_bbI}
\begin{split}
&
\bbI_+\big(\lt_{(I)}\big)=\bbI_+\bsl I_+,\qquad
\cht\big(\lt_{(I)}\big)=\min\big(
(\bbI_+\bsl I_+)\!\sqcup\!\{0\}\big),
\\
&\bbI_\cht\big(\lt_{(I)}\big)=
\big\{e\inn\bbI_\cht\bsl I_\cht\!:
\ell(v_e^+)\!>\!\cht\big(\lt_{(I)}\big)\big\},\\
&
\bbI_-\big(\lt_{(I)}\big)=\bbI_-\bsl I_-
\sqcup\big\{e\inn\bbI_\cht\bsl I_\cht\!:
\ell(v_e^+)\!\le\!\cht\big(\lt_{(I)}\big)\big\}
.
\end{split}
\end{equation}
Intuitively,
the weighted level tree $\lt_{(I)}$ is obtained from $\lt$ by contracting all the edges labeled by $I_-$,
then lifting all the vertices $v$ with $e_v\inn I_\cht$ to the level $\cht$,
and finally contracting all the levels in $I_+$.
Such $\lt_{(I)}$ will be used to describe the local structure of the stack $\mtd{}$ in \S\ref{Sec:Moduli}.

\begin{dfn}\label{Dfn:WLT_equiv}
Two weighted level trees $\lt\eq(\ga,\bfw,\ell)$ and $\lt'\eq(\ga',\bfw',\ell')$ are said to be \ts{equivalent}, written as
$
 \lt\sim\lt',
$
if 
\begin{enumerate}
[leftmargin=*,label=(E\arabic*)]
\item\label{Cond:WLT_equiv_WT} $(\ga,\bfw)\eq(\ga',\bfw')$ as weighted trees;
\item for any $v,w\inn\Ver(\ga)$ satisfying $\ell(v)\eq\ell(w)\!\ge\!\cht(\lt)$, we have
$$\ell'(v)=\ell'(w);$$
\item for any $v,w\inn\Ver(\ga)$ satisfying $\ell(v)\!>\!\ell(w)$ and $\ell(v)\!\ge\!\cht(\lt)$, we have
$$\ell'(v)>\ell'(w).$$
\end{enumerate}
\end{dfn}
It is a direct check that $\sim$ is an equivalence relation on the set of weighted level trees.
Intuitively, this equivalence relation records the relative positions of the vertices {\it above} or {\it in} the level $\cht(\lt)$; see Figure~\ref{Fig:level_tree} for illustration. 

\begin{figure}
\begin{center}
\begin{tikzpicture}{htb}
\draw[dotted]
 (0.4,0)--(5.6,0)
 (0.4,-.8)--(5.6,-.8)
 (0.4,-1.6)--(5.6,-1.6)
 (0.4,-2.4)--(5.6,-2.4);
\draw
 (3,0)--(1.8,-1.6)--(2.4,-2.4)
 (1.8,-1.6)--(1.8,-2.8)
 (1.8,-1.6)--(0.9,-2.8)
 (3,0)--(4,-.8)--(3.2,-2.4)
 (3.6,-1.6)--(4,-2.4)
 (4,-.8)--(5.2,-3.2)
 (5,-2.8)--(4.8,-3.2);
\draw[very thick]
 (3,0)--(1.8,-1.6)
 (3,0)--(4,-.8)
 (3.6,-1.6)--(4,-2.4);
\draw[fill=white]
 (3,0) circle (1.5pt)
 (4,-.8) circle (1.5pt)
 (3.6,-1.6) circle (1.5pt)
 (1.8,-1.6) circle (1.5pt)
 (5,-2.8) circle (1.5pt)
 (8.5,-2.4) circle (1.5pt);
\filldraw
 (3.2,-2.4) circle (1.5pt)
 (4,-2.4) circle (1.5pt)
 (.9,-2.8) circle (1.5pt)
 (1.8,-2.4) circle (1.5pt)
 (1.8,-2.8) circle (1.5pt)
 (2.4,-2.4) circle (1.5pt)
 (5.2,-3.2) circle (1.5pt)
 (4.8,-3.2) circle (1.5pt)
 (8.5,-2.8) circle (1.5pt);
\draw
 (5.8,0) node[right] {\scriptsize{$0=-1[1]=-2[1]=-3[2]=-1^\sharp$}}
 (5.8,-.8) node[right] {\scriptsize{$-1=-2^\sharp$}}
 (5.8,-1.6) node[right] {\scriptsize{$-2=-3[1]=-3^\sharp$}}
 (5.8,-2.4) node[right] {\scriptsize{$\cht=-3$}}
 (4.1,-2) node {\scriptsize{$\se_{-3}$}}
 (2.2,-0.6) node {\scriptsize{$\se_{-2}$}}
 (3.75,-.3) node {\scriptsize{$\se_{-1}$}}
 (3,0) node[above] {\scriptsize{$o=\sv_{-1}^+=\sv_{-2}^+$}}
 (4.2,-0.65) node {\scriptsize{$\sv_{-1}$}}
 (1.55,-1.4) node {\scriptsize{$\sv_{-2}$}}
 (3.4,-1.35) node {\scriptsize{$\sv_{-3}^+$}}
 (4,-2.4) node[below] {\scriptsize{$\sv_{-3}$}}
 (8.5,-2.4) node[right] {\scriptsize{$:\tn{vertex~of~weight}~0$}}
 (8.5,-2.8) node[right] {\scriptsize{$:\tn{vertex~of~positive~weight}$}};
\draw[dashed]
 (8.2,-2.1) rectangle (12.2,-3.1);
\end{tikzpicture}
\end{center}
\caption{A weighted level tree with chosen $\sv_{-1},\sv_{-2},$ and $\sv_{-3}$}\label{Fig:level_tree}
\end{figure}
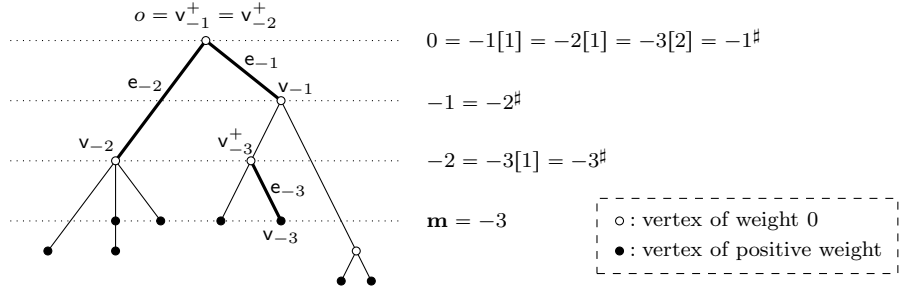  

We denote by
$\sT_\nL^\wt$ the set of the {\it equivalence classes} of the weighted level trees.
There is a natural forgetful map
\begin{equation}\label{Eqn:ff}
\ff:\sT_\nL^\wt\lra\sT_\nR^\wt,\qquad
[\ga,\bfw,\ell]\mapsto(\ga,\bfw),
\end{equation}
which is well defined by Condition~\ref{Cond:WLT_equiv_WT} of Definition~\ref{Dfn:WLT_equiv}.

If $\lt\!\sim\!\lt'$,
then $\cht(\lt)\eq\cht(\lt')$, and there exists a bijection
$$
\phi_{\lt',\lt}:\cP\big(\bbI(\lt)\big)\lra
\cP\big(\bbI(\lt')\big),\qquad
\phi_{\lt',\lt}(I)=\ell'\big(\ell^{-1}(I_+)\big)\sqcup I_{\cht(\lt)}\sqcup I_-\,,
$$
where $\cP(\cdot)$ denotes the power set.
The next lemma follows from  direct check.

\begin{lmm}\label{Lm:T^wt_equiv}
If $\lt\!\sim\!\lt'$,
then 
$\lt_{(I)}\!\sim\!\lt'_{(\phi_{\lt',\lt}(I))}$ for any $I\!\subset\!\bbI(\lt)$.
\end{lmm}

\subsection{Twisted fields}
\label{Subsec:twisted_fields}
For every genus 1 nodal curve $C$, 
its dual graph $\ga_C^\star$ has either a unique vertex $o$ corresponding to the genus 1 irreducible component of $C$ or a unique loop.
In the former case,
$\ga_C^\star$ can be considered as a rooted tree with the root~$o$;
in the latter case,
we contract the loop to a single vertex $o$ and 
obtain a rooted tree  with the root $o$.
Such defined rooted tree is denoted by $\ga_C$ and called the \ts{reduced dual tree} of $C$ (c.f.~\cite[\S3.4]{HL10}).
We call the minimal connected genus 1 subcurve of $C$ the \ts{core} and denote it by $C_o$.
Other irreducible components of $C$ are  smooth rational curves and denoted by $C_v$, $v\inn\Ver(\ga_C)\bsl\{o\}$.
For every incident pair $(v,e)$,
let 
\begin{equation}
\label{Eqn:Nodal}
q_{v;e}\in C_v
\end{equation} 
be the \ts{nodal point} corresponding to the edge $e$.

Let 
$
 \mwt
$
be
the Artin stack of genus 1 stable weighted curves introduced in~\cite[\S2.1]{HL10}.
Here the subscript ``1'' indicating the genus is omitted as per our convention.
The stack $\mwt$ consists of the pairs $(C,\bfw)$ of genus 1 nodal curves $C$ with non-negative weights $\bfw\inn H^2(C,\Z)$,
meaning that $\bfw(\Sigma)\!\ge\!0$ for all irreducible $\Sigma\!\subset\!C$.
Here $(C,\bfw)$ is said to be \ts{stable} if every rational irreducible component of weight $0$ contains at least three nodal points.
The weight of the core $\bfw(C_o)$ is defined as the sum of the weights of all irreducible components of the core.

Every $(C,\bfw)\inn\mwt$ uniquely determines a function
$$\bfw :\Ver(\ga_C)\lra\Z_{\ge 0},\qquad
v\mapsto \bfw(C_v),$$
which makes the pair $(\ga_C,\bfw)$ a weighted tree,
called the \ts{weighted dual tree}.
Thus, the stack $\mwt$ can  be stratified as 
\begin{equation}\label{Eqn:Mwt_strata}
 \mwt=\bigsqcup_{\tr\in\sT_\nR^\wt}
 \!\!\!\mwt_{\tr}~=
 \bigsqcup_{\tr\in\sT_\nR^\wt}
 \!\!\!\big\{\,
 (C,\bfw)\inn\mwt\!:(\ga_C,\bfw)\eq\tr \,
 \big\}.
\end{equation}
If the sum of the weights of all vertices is fixed,
the stability condition of $\mwt$ then guarantees there are only finitely many $\tr\inn\sT_\nR^\wt$ so that $\mwt_\tr$ is non-empty.

Given $\tr\eq(\ga,\bfw)\inn\sT_\nR^\wt$
and $e\inn\Edg(\tr)$,
let $$L_e^\pm\lra\mwt_\tr$$ be the line bundles whose fibers over a weighted curve $(C,\bfw)$ are the tangent vectors of the irreducible components $C_{v_e^\pm}$ at the nodal points $q_{v_e^\pm;e}$, respectively.
We take
\begin{equation}\label{Eqn:Le}
L_e= L_e^+\otimes L_e^-
\lra \mwt_\tr,\qquad
L_e^\succeq=
\!\!\!\!\!\!\bigotimes_{
\begin{subarray}{c}
 e'\in\Edg(\tr),\,
 e'\succeq e
\end{subarray}
}\!\!\!\!\!\!\!\!\!\!
L_{e'}\lra\mwt_\tr.
\end{equation} 

For any direct sum of line bundles $V\eq\oplus_mL_m'$ (over any base),
we write
$$
 \mathring\P(V):=
 \big\{
 \big(x,[v_m]\big)\inn\P(V):
 v_m\!\ne\!0~\forall\,m
 \big\}.
$$ 
For any morphisms $M_1,\ldots, M_k\!\lra\! S$,
we write
$$
 \prod_{1\le i\le k}
 \!(M_i/S) \,
 :=
 M_1\times_S M_2\times_S\cdots
 \times_S M_k.
$$ 
With notation as above,
given $\tr\inn\sT_\nR^\wt$ and $[\lt]\eq\big[\tr,\ell\big]\inn\sT_{\nL}^\wt$,
let  
\begin{equation}\begin{split}
\label{Eqn:MtdStrat_and_sEcT}
 \varpi:\mtd{[\lt]}=
 \bigg(
 \prod_{i\in\bbI_+(\lt)}\!\!\!
  \Big\lgroup\!
  \Big(\mathring\P\big(\!
   \bigoplus_{\!
   \begin{subarray}{c}
    e\in\Edg(\lt),\,
    \ell(v_e^-)=i
   \end{subarray}
   }\!\!\!\!\!\!\!\!\!\!\!
   L_e^\succeq\;
  \big)\!\Big)\Big/\mwt_\tr\Big\rgroup\!\!
 \bigg)
 &\lra\mwt_{\tr}\,,\\
 \sE_{[\lt]}=
 \bigg(
 \prod_{i\in\bbI_+(\lt)}
 \!\!\!
 \Big\lgroup
 \Big(\P\big(\!
 \bigoplus_{e\in\fE_i}L_e^\succeq
 \big)\!\Big)\Big/\mwt_\tr\Big\rgroup\!\!
 \bigg)
 &\lra\mwt_{\tr}\,,
\end{split}\end{equation}
where $\bbI_+(\lt)$, $L_e^\succeq$, and $\fE_i$ are as in~(\ref{Eqn:bbI}),~(\ref{Eqn:Le}), and~(\ref{Eqn:fE_i}), respectively.
It is straightforward that both bundles in~(\ref{Eqn:MtdStrat_and_sEcT}) are independent of the choice of the weighted level tree $\lt$ representing $[\lt]$.
Since
$$
 \big\{e\inn\Edg(\lt)\!:\ell(v_e^-)\eq i
 \big\}\subset\fE_i
 \qquad\forall~i\inn\bbI_+(\lt),
$$
we see that
$\mtd{[\lt]}$ is a subset of $\sE_{[\lt]}$.
In addition,
since each stratum $\mwt_\tr$ is an algebraic stack,
so are $\mtd{[\lt]}$ and $\sE_{[\lt]}$.

Using~(\ref{Eqn:MtdStrat_and_sEcT}) and~(\ref{Eqn:Mwt_strata}),
we define
\begin{equation}\label{Eqn:MtdStrata}
\mtd{}:=\bigsqcup_{[\lt]\in\sT_\nL^\wt}\!\!
\mtd{[\lt]}~\stackrel{\varpi}{\lra}~
\mwt.
\end{equation}
This is the set-theoretic definition of the proposed stack $\mtd{}$ as well as the forgetful map in Theorem~\ref{Thm:Main}.
For any $x\inn\mwt_\tr$, 
the points of the fiber $\mtd{[\lt]}\big|_x$ are called the \ts{twisted fields} over $x$.

\begin{rmk}
By~(\ref{Eqn:MtdStrata}),
$\ti M_1^\tf(\P^n,d)$ in Corollary~\ref{Crl:Main} consists of the tuples
$$
 \big(\,C,\,\bu,\,[\lt],\,\ud\eta\,\big),
$$
where $(C,\bu)$ are stable maps in $\ov M_1(\P^n,d)$,
$[\lt]$ are the equivalence classes of weighted level trees satisfying $\ff[\lt]\eq\big(\ga_C,c_1(\bu^*\sO_{\P^n}(1))\big)$,
and $\ud\eta$ are twisted fields over $\big(C,c_1(\bu^*\sO_{\P^n}(1))\big)$.
\end{rmk}

\section{The stack structure of $\mtd{}$}
\label{Sec:Moduli}
In \S\ref{Sec:Moduli}, we show $\mtd{}$ is naturally a  smooth Artin stack and describe its universal family.

\subsection{Twisted charts}\label{Subsec:Charts_of_Mtf}
We first fix
$[\lt]\eq\big[\ga,\bfw,\ell\big]\inn\sT_{\nL}^\wt
$ and $x\inn\mtd{[\lt]}$,
and write
$$
 \tr=\ff[\lt]=(\ga,\bfw)\in\sT_\nR^\wt,\qquad
 (C,\bfw)=\varpi(x)\in\mwt_\tr.
$$
Since $\mwt$ is smooth,
we take an affine smooth chart 
$$
 \cV=\cV_{\varpi(x)}\lra\mwt
$$ containing $(C,\bfw)$.

As in~\cite[\S4.3]{HL10} and~\cite[\S2.5]{HLN},
there exists a set of regular functions 
\begin{equation}\label{Eqn:modular_parameters}
 \ze_e\in\Ga(\sO_\cV)\quad\tn{with}\quad e\in\Edg(\ga),
\end{equation}
called the \ts{modular parameters},
so that for each $e\inn\Edg(\ga)$, the locus 
$$\cZ_e=\{\ze_e\eq 0\}\subset\cV$$ is the irreducible smooth Cartier divisor on $\cV$ where the node labeled by $e$ is not smoothed.
For any $I\!\subset\!\bbI\eq\bbI(\lt)$, let
\begin{align*}
 \cV_{(I)}^\circ
 &:=
 \big\{\,
 \ze_{e'}\!\ne\!0\!: e'\inn\big(\Edg(\ga)\bsl\Edg(\ga_{(I)})\big)\,\big\}&\subset\cV,&
 \\
 \cV_{\tr_{(I)}}\!
 &:=\cV_{(I)}^\circ\cap\big\{\,
 \ze_e\eq 0\!: e\inn\Edg\big(\ga_{(I)}\big)\,\big\}& \subset\cV_{(I)}^\circ\subset\cV.&
\end{align*}
Then, $\cV_{(I)}^\circ$ is an open subset of $\cV$. Shrinking $\cV$ if necessary,
we see that
\begin{equation}
\begin{split}
\label{Eqn:V_(I)}
\cV_{\tr_{(I)}}\in\pi_0\big(
\mwt_{\tr_{(I)}}\cap\cV
\big),
\end{split}
\end{equation}
where $\pi_0$ denotes the set of the connected components.
In particular, $\cV_{\tr_{(I)}}$ can be considered as a smooth chart of the stratum $\mwt_{\tr_{(I)}}$.
Rigorously, the sets $\bbI$ and $I$ depend on the choice of the weighted level tree $\lt$ representing $[\lt]$,
however, $\cV_{(I)}^\circ$ and $\cV_{\tr_{(I)}}$ are independent of such choice;
see Lemma~\ref{Lm:T^wt_equiv}.
 
Given a set of the modular parameters as in~(\ref{Eqn:modular_parameters}),
we may extend it to a set of local parameters on~$\cV$ centered at $(C,\bfw)$:
\begin{equation}\label{Eqn:local_parameters}
 \{\ze_e\}_{e\in\Edg(\ga)}\cup
 \{\vs_j\}_{j\in J}\qquad
 \tn{with}\quad
 (C,\bfw)=\ud 0:=(0,\ldots,0)\,,
\end{equation}
where $J$ is a finite set.
We do not impose other conditions on $\vs_j$.
 
For each $e\inn\Edg(\ga)$,
we set
$$
 \prt_{\ze_e}:=
 (\tn d\ze_e)^\vee
 \in\Ga(\cV; T\mwt).
$$
 
\begin{lmm}
\label{Lm:Prt_e}
For every $I\!\subset\!\bbI$ and every $e\inn\Edg(\ga_{(I)})$,
the restriction of $\prt_{\ze_e}$ to $\cV_{\tr_{(I)}}$ is a nowhere vanishing section of the restriction of the line bundle~$L_e$ in~(\ref{Eqn:Le}) to $\cV_{\tr_{(I)}}$.
\end{lmm}

\begin{proof}
Since the restriction of $\tn d\ze_e$ to $\cZ_e\eq\{\ze_e\eq 0\}$ ($\supset\!\cV_{\tr_{(I)}}$) is identically zero,
we observe that the restriction of $\prt_{\ze_e}$ to $\cZ_e$ is a nowhere vanishing section of the normal bundle of $\cZ_e$.
It is a well-known fact of  the moduli of curves that the normal bundle of $\cZ_e$ is $L_e$;
see~\cite[Proposition~3.31]{HM}.
\end{proof}

For each level $i\inn\bbI_+\eq\bbI_+(\lt)$,
we choose a special vertex
\begin{equation}\label{Eqn:sv_i}
\sv_i\inn\Ver(\ga)\qquad\tn{s.t.}\qquad\ell(\sv_i)\eq i.
\end{equation}
We then denote by $\se_i$, $\se_i^+$, and $\sv_i^+$ respectively the edges and the vertex satisfying
\begin{equation}\label{Eqn:se_i}
 \se_i\eq e_{\sv_i},\qquad
 \sv_i^+=v_{\se_i}^+\,,\qquad
 \se_i^+=e_{\sv_i^+}.
\end{equation}
Each $i\inn\bbI_+$ determines a strictly increasing sequence
\begin{equation}\label{Eqn:i[h]}
 i[0]:= i~<~
 i[1]:= \ell(\sv_i^+)~<~
 i[2]:= \ell(\sv_{i[1]}^+)~<~\cdots\,.
\end{equation}
We would like to remark that $i[1]$ and $i^\sharp$ in~(\ref{Eqn:i^sharp}) need not to be the same;
see Figure~\ref{Fig:level_tree} for illustration.
This sequence is finite,
as there is a unique step~$h$ satisfying $i[h]\eq 0$.

By Lemma~\ref{Lm:Prt_e},
there exist $\la_e\inn\A$ with $e\inn\wh\Edg(\lt)$ so that
the fixed $x\inn\mtd{[\lt]}$ over $(C,\bfw)$ can uniquely be written as
\begin{equation}\label{Eqn:x_local_expression}
\begin{split}
 & x=
 \Big(\,
 \ud 0\;;\;\prod_{i\in\bbI_+}\Big[\la_e\!\cdot\!
 \big(\bigotimes_{\fe\succeq e}\!\prt_{\ze_{\fe}}|_{\ud 0}\big)
 : \ell(e)\eq i\Big]\,
 \Big),\qquad\tn{where}
 \\
 & 
 \la_e\!\ne\!0
 \ \ \forall\,e\inn\wh\Edg(\lt)\bsl\bbI_\cht,\quad
 \la_{\se_i}\eq 1\ \ \forall\,i\inn\bbI_+,
 \quad
 \la_e\eq 0\ \ \forall\,e\inn\bbI_\cht.
\end{split}
\end{equation}
Let 
$$
\fU_x\subset \A^{\bbI_+}\times \A^{\wh\Edg(\lt)\bsl\{\se_i:{i\in\bbI_+}\}}\times
\A^{\bbI_-}\times
\A^{J}
$$
be an open subset containing the point
\begin{equation}\label{Eqn:y_x}
 y_x:=
 \big(\ud 0,\,
 (\la_e)_{e\in\wh\Edg(\lt)\bsl\{\se_i:{i\in\bbI_+}\}}\,,
 \ud 0,\,
 \ud 0\big).
\end{equation}
The coordinates on $\fU_x$ are denoted by
\begin{equation}\label{Eqn:ti_cV}
\big(
(\ve_i)_{i\in\bbI_+},
(u_e)_{e\in\wh\Edg(\lt)\bsl\{\se_i:{i\in\bbI_+}\}},
(z_e)_{e\in\bbI_-},
(w_j)_{j\in J}
\big).
\end{equation}

For any $I\!\subset\!\bbI$,
we set
\begin{align*}
 \fU_{x;(I)}^\circ
 &\eq
 \big\{\;
 \ve_i\!\ne\!0~\forall\,i\inn I_+\;;~
 u_e\!\ne\!0~\forall\,e\inn I_\cht\;;~
 z_e\!\ne\!0~\forall\,e\inn I_-\,
 \big\}
 \subset\fU_x,
 \\
 \fU_{x;[\lt_{(I)}]}
 &\eq
 \fU_{x;(I)}^\circ\!\cap\!
 \big\{\,
 \ve_i\eq 0~\forall\,i\inn\bbI_+\!\bsl I_+;~
 u_{e}\eq 0~\forall\,e\inn \bbI_{\cht}\!\bsl I_\cht;~
 z_{e}\eq 0~\forall\,e\inn\bbI_-\!\bsl I_-
 \big\}.
\end{align*}
This gives rise to a stratification
\begin{equation}\label{Eqn:ti_V_strata}
\fU_x=\bigsqcup_{I\subset\bbI}\,\fU_{x;[\lt_{(I)}]}.
\end{equation}
 We remark that neither $\fU_x$ nor its stratification~(\ref{Eqn:ti_V_strata}) depends on the choice of the weighted level tree $\lt$ representing $[\lt]$,
 even though the sets $\bbI$ and $I$ depend on such choice.
 We also notice that
 $\fU_{x;(I)}^\circ$ is an open subset of $\fU_x$,
 but the strata $\fU_{x;[\lt_{(I)}]}$ are not open unless $I\eq \bbI$.

For each $i\inn\bbI_+$, we take
$$
u_{\se_i}:= 1.
$$
By~(\ref{Eqn:x_local_expression}), shrinking $\fU_x$ if necessary, we have 
\begin{equation}\label{Eqn:al_e_nonzero}
u_e\inn\Ga\big(\sO_{\fU_x}^*\big)\qquad
\forall~e\inn\wh\Edg(\lt)\bsl\bbI_\cht.
\end{equation}
With the local parameters
$\ze_e$ and~$\vs_j$ as in~(\ref{Eqn:local_parameters}),
we construct a morphism
$$
\th_x :\fU_x\lra\cV\quad(\lra\mwt)
$$ given by
\begin{equation}\begin{split}\label{Eqn:theta_x}
&\th_x^*\ze_e=
\frac{u_e\cdot u_{\se_{\ell(e)}^+}\!\!\cdot u_{\se_{\ell(e)[1]}^+}\cdots}
{ u_{e_{v_e^+}}\!\!\cdot u_{\se^+_{\ell(v_e^+)}}\!\!\cdot u_{\se_{\ell(v_e^+)[1]}^+}\cdots}
\cdot \!
\prod_{
\begin{subarray}{c}
i\in\lbrp{\ell(e),\ell(v_e^+)}
\end{subarray}}
\!\!\!\!\!\!\!\!\ve_i
\qquad\forall~e\inn\wh\Edg(\lt);
\\
&\th_x^*\ze_e= z_e\quad\forall~e\inn\bbI_-\eq\Edg(\lt)\bsl\wh\Edg(\lt);
\qquad
\th_x^*\vs_j=w_j\quad\forall~j\inn J.
\end{split}\end{equation}
The numerator and the denominator in the first line of~(\ref{Eqn:theta_x}) are both finite products,
because~(\ref{Eqn:i[h]}) is always a finite sequence.

For any $I\!\subset\!\bbI$,
it follows from~(\ref{Eqn:al_e_nonzero}) and~(\ref{Eqn:theta_x}) that
\begin{equation}\label{Eqn:th_{x;(I)}}
 \th_x\big(\fU_{x;{(I)}}^\circ\big)\subset\cV_{{(I)}}^\circ,\qquad
 \th_x\big(\fU_{x;[\lt_{(I)}]}\big)\subset\cV_{\tr_{(I)}},
\end{equation}
where $\cV_{{(I)}}^\circ$ and $\cV_{\tr_{(I)}}$ are described before~(\ref{Eqn:V_(I)}).
 
Fix $I\!\subset\!\bbI$ ($I$ may be empty).
With 
$$
 [\lt_{(I)}]=[\tau_{(I)},\ell_{(I)}]\in\sT_\nL^\wt
$$ as in~(\ref{Eqn:t(I)}), $\mtd{[\lt_{(I)}]}$ as in~(\ref{Eqn:MtdStrata}),
and the chart $\cV_{\tr_{(I)}}\!\lra\!\mwt_{\tr_{(I)}}$ as in~(\ref{Eqn:V_(I)}),
let 
$$
\Phi_{x;(I)} :\fU_{x;[\lt_{(I)}]}\lra
\cV_{\tr_{(I)}}\!\times_{\mwt_{\tr_{(I)}}}\!
\mtd{[\lt_{(I)}]}\quad
\big(\lra
\mtd{[\lt_{(I)}]}\big)
$$
be the morphism so that for any $y\inn\fU_{x;\lt_{(I)}}$, 
\begin{equation}\label{Eqn:Phi_x_(I)}
\Phi_{x;(I)}(y)\eq \Big(
\th_x(y);
\prod_{i\in\bbI_+\bsl I_+}\!\!\!
\Big[
\mu_{e;i;I}(y)\!\cdot\!
\big(\!
\bigotimes_{\fe\succeq e,~
\lbrp{\ell(\fe),\ell(v_{\fe}^+)}\not\subset I}
\hspace{-.4in}
\prt_{\ze_{\fe}}\big|_{\th_x(y)}
\big):e\inn\fE_i
\Big]
\Big),
\end{equation} 
where
\begin{align*}
\mu_{e;i;I} &=\!
\bigg(\!u_e\!\cdot\!
\frac{
  u_{\se_{\ell(e)}^+}\!\!\cdot\! u_{\se_{\ell(e)[1]}^+}\!\cdots
}{
u_{\se_{i}^+}\!\cdot u_{\se_{i[1]}^+}\!\cdots
}\bigg)\!
\Bigg(
\frac{
 \prod_{
\begin{subarray}{c}
\fe\succ \se_i;\,
\lbrp{\ell(\fe),\ell(v_{\fe}^+)}\subset I
\end{subarray}}\,
 \th_x^*\ze_{\fe}
 }{ 
\prod_{
\begin{subarray}{c}
 e'\succ e;\,
\lbrp{\ell(e'),\ell(v_{e'}^+)}\subset I
\end{subarray}}\,
\th_x^*\ze_{e'}
}\Bigg)\!
\Big(\!\!
\prod_{
\begin{subarray}{c}
h\in \lbrp{\ell(e),i}
\end{subarray}}
\!\!\!\!\!\ve_h\Big)
\end{align*}
for all $i\inn\bbI_+\bsl I_+$ and $e\inn\fE_i$.
Similar to~(\ref{Eqn:theta_x}),
the products in the first pair of parentheses above are both finite products.

By~(\ref{Eqn:th_{x;(I)}}), the description of $\cV_{(I)}^\circ$ above~(\ref{Eqn:V_(I)}), and~(\ref{Eqn:edges_cntrd}),
we see that
$$\mu_{e;i;I}\in\Ga\big(\sO_{\fU_{x;(I)}^\circ}\big)
\qquad
\forall\,I\!\subset\!\bbI,\,i\inn\bbI_+\bsl I_+,~e\inn\fE_i.
$$
Moreover, by~(\ref{Eqn:al_e_nonzero}),
\begin{equation}\label{Eqn:mu_vanishing}
 \mu_{e;i;I}|_{\fU_{x;[\lt_{(I)}]}}
 \begin{cases}
 =0 & \tn{if}\ \ \ell_{(I)}(v_e^-)<i,\\
 \in \Ga\big(\sO^*_{\fU_{x;[\lt_{(I)}]}}\big) 
 & \tn{if}\ \ \ell_{(I)}(v_e^-)=i.
 \end{cases}
\end{equation}
This, along with~(\ref{Eqn:th_{x;(I)}}), Lemma~\ref{Lm:Prt_e}, and (\ref{Eqn:MtdStrat_and_sEcT}), 
implies~$\Phi_{x;(I)}$ is well-defined.

The morphisms $\Phi_{x:(I)}$, $I\!\subset\!\bbI$, together determine 
$$\Phi_x:\fU_x\lra\mtd{},\qquad
\Phi_x(y)=\Phi_{x;(I)}(y)\quad\tn{if}~y\in\fU_{x;[\lt_{(I)}]}.$$
We remark that $\Phi_{x;(I)}$ and $\Phi_x$ are also independent of the choice of the weighted level tree $\lt$ representing $[\lt]$.
Moreover,
we observe that
\begin{equation}
\label{Eqn:Phi_onto}
 \Phi_x(y_x)=\Phi_{x;(\emptyset)}(y_x)=x
 \qquad
 \forall\,x\in\mtd{},
\end{equation}
where $y_x\inn\fU_{x}$ is given in~(\ref{Eqn:y_x}).

A priori $\Phi_x$ is just a {\it map}, for the set-theoretic definition~(\ref{Eqn:MtdStrata})  of $\mtd{}$ does not describe its stack structure, although each $\mtd{[\lt']}$ is a stack.
In~\S\ref{Subsec:transition_maps},
we will show such $\Phi_x$ patch together to endow $\mtd{}$ with a smooth stack structure.
Each $\Phi_x$ will hereafter be called a \ts{twisted chart  centered at~$x$ (lying over $\cV\!\lra\!\mwt$)},
although rigorously it becomes a {\it chart} of $\mtd{}$ only after Corollary~\ref{Crl:MtdSmooth} is established. 

\begin{lmm}
\label{Lm:Phi_x_inj}
For every $I\!\subset\!\bbI$,
$
 \Phi_{x;(I)}\!:\fU_{x;[\lt_{(I)}]}\!\lra\mtd{[\lt_{(I)}]}
$
of~(\ref{Eqn:Phi_x_(I)})
is an isomorphism to an open subset of $\mtd{[\lt_{(I)}]}$.
\end{lmm}

\begin{proof}
For any $i\inn\bbI_+(\lt_{(I)})\eq\bbI_+\bsl I_+$,
notice that every edge in $\fE_i$ of the weighted level tree $\lt$ is not contracted in the construction of~$\lt_{(I)}$ (c.f.~(\ref{Eqn:edges_cntrd})).
Thus,
$$ 
 \fE_i\subset\Edg\big(\lt_{(I)}\big)\qquad
 \forall~i\inn\bbI_+(\lt_{(I)}).
$$
In particular,
the edges $\se_i$, $i\inn\bbI_+(\lt_{(I)})$, can be used as the special edges of~$\lt_{(I)}$.
For conciseness, let
\begin{equation*}
\begin{split}
 \bfE[\lt_{(I)}] :\!\!&=
 \wh\Edg(\lt_{(I)})\big\bsl\!\left(\big\{\se_i\!:i\in\bbI_+(\lt_{(I)})\big\}\sqcup\bbI_\cht(\lt_{(I)})\right)\\
 &=\!\!\!\!
 \bigsqcup_{i\in \bbI_+(\lt_{(I)})}
 \!\!\!\!\!\!
 \big\{e\inn\Edg(\lt_{(I)})\!:
 \ell_{(I)}(v_e^-)\eq i,\,e\!\ne\!\se_i
 \big\}
 \quad 
 \big(\subset\!\wh\Edg(\lt_{(I)})\!\subset\!\wh\Edg(\lt)\big);
\end{split}
\end{equation*}
see~(\ref{Eqn:t_(I)_bbI}) for notation.

Let $\{\ze_e\}_{e\in\Edg(\ga)}\!\cup\!\{\vs_j\}_{j\in J}$ be a set of the local parameters on $\cV$ centered at $\varpi(x)$ as in~(\ref{Eqn:local_parameters}).
By the definition of $\cV_{\tr_{(I)}}$ above~(\ref{Eqn:V_(I)}),
\begin{equation}\label{Eqn:V_(I)_parameters}
 \{\ze_e\}_{
 e\in\Edg(\lt)\bsl\Edg(\lt_{(I)})
 }
 \sqcup
 \{\vs_j\}_{j\in J}
\end{equation}
is a set of local parameters of $\cV_{\tr_{(I)}}$.

Recall that there exist $\la_e\inn\A^*$, $e\inn\wh\Edg(\lt)$, such that
$$
  x=
 \Big(\,
 \ud 0\;;\;\prod_{i\in\bbI_+}\Big[\la_e\!\cdot\!
 \big(\bigotimes_{\fe\succeq e}\prt_{\ze_{\fe}}|_{\ud 0}\big)
 : \ell(v_e^-)\eq i\Big]\,
 \Big)
$$ 
as in~(\ref{Eqn:x_local_expression}).
Let $U_{x;\bfE[\lt_{(I)}]}$ be  an open subset of $(\A^*)^{\bfE[\lt_{(I)}]}$ such that
$$
 (\la_e)_{e\in\bfE[\lt_{(I)}]}\,\in\,
 U_{x;\bfE[\lt_{(I)}]}\,\subset\,
 (\A^*)^{\bfE[\lt_{(I)}]}.
$$
The coordinates of $U_{x;\bfE[\lt_{(I)}]}$ are denoted by
$$
 (\,\mu_e\,)_{e\in\bfE[\lt_{(I)}]}.
$$
In addition, we set 
$$
 \mu_{\se_i}= 1\quad
 \forall~i\inn\bbI_+(\lt_{(I)}),\qquad
 \mu_e=0\quad\forall~e\inn\bbI_\cht(\lt_{(I)}).
$$
Thus, the function $\mu_e$ is defined for all $e\inn\wh\Edg(\lt_{(I)})$, and is nowhere vanishing on~$U_{x;\bfE(\lt_{(I)})}$ for all $e\inn\wh\Edg(\lt_{(I)})\bsl\bbI_\cht(\lt_{(I)})$.

The smooth chart $\cV_{\tr_{(I)}}\!\lra\!\mwt_{\tr_{(I)}}$ in (\ref{Eqn:V_(I)}) induces a smooth chart
$$
 \cU'_{x;[\lt_{(I)}]}:=\cV_{\tr_{(I)}}\times U_{x;\bfE[\lt_{(I)}]}
\lra\mtd{[\lt_{(I)}]}
$$
given by
\begin{equation*}\begin{split}
 &\big(\,
 \mathfrak z,\,
 (\mu_e)_{e\in\bfE(\lt_{(I)})}
 \big)
 \mapsto
 \Big(\mathfrak z\,,
 \prod_{i\in\bbI_+(\lt_{(I)})}
 \!\!\!\!
 \big[
 \mu_e\big(\!
 \bigotimes_{
 \fe\in\Edg(\lt_{(I)}),\,\fe\succeq e}
 \hspace{-.25in}
 \prt_{\ze_\fe}|_{\mathfrak z}
 \big):
 \ell_{(I)}\!(v_e^-)\eq i
 \big]
 \Big).
\end{split}\end{equation*}
We will construct a morphism 
\begin{equation}\label{Eqn:Psi_{x;(I)}}
 \Psi_{x;(I)}:\cU'_{x;[\lt_{(I)}]}\lra\fU_{x;[\lt_{(I)}]}
\end{equation}
such that $\Phi_{x;(I)}\!\circ\!\Psi_{x;(I)}$ and $\Psi_{x;(I)}\!\circ\!\Phi_{x;(I)}$ are both the identity morphisms,
which will then establish Lemma~\ref{Lm:Phi_x_inj}.

Given $\big(
 \mathfrak z,
 (\mu_e)_{e\in\bfE(\lt_{(I)})}
 \big)\inn\cU'_{x;\lt_{(I)}},$
we denote its image by
$$
 y:=\Psi_{x;(I)}\big(
 \mathfrak z,
 (\mu_e)_{e\in\bfE(\lt_{(I)})}
 \big)\in\fU_{x;[\lt_{(I)}]},
$$
which is to be constructed.
With the coordinates on $\fU_x$ as in~(\ref{Eqn:ti_cV}), we set
\begin{equation}\label{Eqn:y_other'}
 z_e(y)\eq \ze_e(\fz)\ \ \forall\,e\inn\bbI_-,\qquad
 w_j(y)\eq \vs_j(\fz)\ \ \forall\,j\inn J.
\end{equation}
By~(\ref{Eqn:V_(I)}),
we see that
\begin{equation}\begin{split}\label{Eqn:y_other}
 z_e(y)\eq\ze_e(\fz)\eq 0
 \qquad
 &\Longleftrightarrow\qquad
 e\in\bbI_-\bsl I_-\;\big(\subset\bbI_-(\lt_{(I)})\big).
\end{split}\end{equation}
The rest of the coordinates of $y$ are much more complicated;
we describe them by induction over the levels in $\bbI_+\eq\bbI_+(\lt)$.
More precisely,
we will show that
{\it
$\ve_i(y)$ with $i\inn\bbI_+$ and $u_e(y)$ with $e\inn\wh\Edg(\lt)\bsl\{\se_i\!:i\inn\bbI_+\}$  are all rational functions in $\ze_{e'}(\fz)$ and $\mu_{e''}$, satisfying }
\begin{equation}\label{Eqn:ve_u_e_inductive}
 \big\lgroup
 \ve_i(y)\eq 0~\Leftrightarrow~
 i\inn \bbI_+\bsl I_+
 \big\rgroup
 \quad
 \tn{and}\quad
 \big\lgroup
 u_e(y)\eq 0~\Leftrightarrow~
 e\inn\bbI_\cht\bsl I_\cht
 \big\rgroup.
\end{equation}
In particular,
(\ref{Eqn:ve_u_e_inductive}) and~(\ref{Eqn:y_other}) imply $y\inn\fU_{x;[\lt_{(I)}]}$,
i.e.~$\Psi_{x;(I)}$ is well-defined.

The base case of the induction is for the level
$$i_0:=\max \bbI_+(\lt).$$
We take
\begin{equation}\label{Eqn:v_e_base}
 \ve_{i_0}(y)=\ze_{\se_{i_0}}\!(\fz).
\end{equation}
By~(\ref{Eqn:V_(I)}), we see that $\ve_{i_0}(y)$ satisfies~(\ref{Eqn:ve_u_e_inductive}).
We take
$$
 u_{\se_{i_0}}(y)=1.
$$
For any $e\!\ne\!\se_{i_0}$ with $\ell(e)\eq i_0$, we set
\begin{equation}\label{Eqn:u_e_base}
 u_e(y)=
 \begin{cases}
  \mu_e
  \quad
 &\tn{if}~i_0\!\not\in\! I_+~\big(\,\tn{i.e.}~i_0\inn\bbI_+(\lt_{(I)}),~e\inn\bfE(\lt_{(I)})\,\big)
 \\
 \frac{\ze_e(\fz)}{\ze_{\se_{i_0}}\!\!(\fz)}
 &\tn{if}~i_0\inn I_+~\big(\,\tn{i.e.}~\se_{i_0}\inn\Edg(\lt)\bsl\Edg(\lt_{(I)})\,\big)
 \end{cases}.
\end{equation}
If $i_0\inn I_+$,
then by~(\ref{Eqn:V_(I)_parameters}) and~(\ref{Eqn:V_(I)}),
we have 
$\ze_e(\fz)\eq 0$ if and only if ($\cht\eq i_0$ and) $e\inn\bbI_\cht\bsl  I_\cht$.
If $i_0\!\not\in\!I_+$,
then $\mu_e\eq 0$ if and only if ($\cht\eq i_0$ and) $e\inn\bbI_\cht\bsl  I_\cht$.
We thus conclude that $u_e(y)$ satisfies~(\ref{Eqn:ve_u_e_inductive}) for all 
$e\inn\wh\Edg(\lt)$ with $\ell(e)\eq i_0.$
Moreover, such $\ve_{i_0}(y)$ and $u_e(y)$ are obviously rational functions in $\ze_{e'}(\fz)$ and $\mu_{e''}$. 
Hence, the base case is complete.

Next, for any $i\inn\bbI_+$, assume that
all $\ve_k(y)$ with $k\!>\!i$ and all $u_e(y)$ with $e\inn\wh\Edg(\lt)$ and $\ell(e)\!>\!i$ have been expressed as rational functions in $\ze_{e'}(\fz)$ and $\mu_{e''}$, satisfying~(\ref{Eqn:ve_u_e_inductive}).

For the level $i$, we first construct $\ve_i(y)$.
The construction is subdivided into three cases.

\textbf{Case 1}.
If $i\!\not\in\!I_+$,
then set
\begin{equation*}\label{Eqn:ve_i_induc_0}
 \ve_{i}(y)=0.
\end{equation*}
Obviously this satisfies~(\ref{Eqn:ve_u_e_inductive}).

\textbf{Case 2}.
If $i\inn I_+$ and $\lbrp{i,i[1]}\!\not\subset\!I_+$,
then $\se_i\inn\bfE(\lt_{(I)})$,
hence $\mu_{\se_i}\!\ne\!0$.
Let
$$
 \wh i:=\min\big(\lbrp{i,i[1]}\bsl I_+\big)
 ~ \in \bbI_+(\lt_{(I)}).
$$
Intuitively, $\wh i$ is the level containing the image of $\sv_i$ in $\lt_{(I)}$.
Thus, 
$$
 \ell_{(I)}(\se_i)=\wh i.
$$
Let $\ve_i(y)$ be given by
\begin{equation}\label{Eqn:ve_i_induc_1}
\begin{split}
 \mu_{\se_i}\eq 
 \ve_i(y)
 \!\cdot\!
 \frac{u_{\se_i^+}\!(y)\!\cdot\! u_{\se_{i[1]}^+}\!\!(y)\cdots}
 {u_{\se_{\wh i}^+}\!(y)\!\cdot\! u_{\se_{\wh i[1]}^+}\!\!(y)\cdots}
 \!\cdot\! 
 \frac{\prod_{\fe\succ \se_{\wh i},\,\lbrp{\ell(\fe),\ell(v_\fe^+)}_\lt\subset I}\ze_{\fe}(\fz)}
 {\prod_{e'\succ \se_i,\,\lbrp{\ell(e'),\ell(v_{e'}^+)}_\lt\subset I}\ze_{e'}\!(\fz)}\!\cdot\!\!\!
 \prod_{
 h\in\lprp{i,\wh i}_\lt}\!\!\!
 \ve_h(y).
\end{split}
\end{equation}
The inductive assumption implies that all $\ve_h(y)$ with $h\inn\lprp{i,\wh i}_\lt$, as well as all $u_{\se_{i[h]}^+}\!\!(y)$ and $u_{\se_{\wh i[h]}^+}\!\!(y)$ with $h\!\ge\!0$, are non-zero and are rational functions in $\ze_{e'}(\fz)$ and $\mu_{e''}$. 
By~(\ref{Eqn:V_(I)_parameters}) and~(\ref{Eqn:edges_cntrd}),
we also see that all $\ze_\fe(\fz)$ and $\ze_{e'}\!(\fz)$ in~(\ref{Eqn:ve_i_induc_1}) are non-zero.
Therefore, such defined 
$\ve_i(y)$ is a rational function in $\ze_{e'}(\fz)$ and $\mu_{e''}$, satisfying~(\ref{Eqn:ve_u_e_inductive}).

\textbf{Case 3}.
If  $\lbrp{i,i[1]}\!\subset\!I_+$ (hence $i\inn I_+$),
then we see $\se_i\inn\Edg(\lt)\bsl\Edg(\lt_{(I)})$;
c.f.~(\ref{Eqn:edges_cntrd}).
Intuitively, this means $\se_i$ is contracted in the construction of $\lt_{(I)}$.
By the description of $\cV_{\tr_{(I)}}$ above~(\ref{Eqn:V_(I)}),
we see that
$\ze_{\se_i}(\fz)\!\ne\!0$.
Let $\ve_i(y)$ be given by
\begin{equation}\begin{split}
\label{Eqn:ve_i_induc_2}
 \ze_{\se_i}(\fz)&=
 \ve_i(y)\cdot
 \!\!\!
 \prod_{
 h\in\lprp{i,\wh i}_\lt}\!\!\!
 \ve_h(y).
\end{split}\end{equation}
Mimicking the argument in Case 2,
we conclude that $\ve_i(y)$ is a rational function in $\ze_{e'}(\fz)$ and $\mu_{e''}$, satisfying~(\ref{Eqn:ve_u_e_inductive}).

Next, we construct $u_e(y)$ for $e\inn\wh\Edg(\lt)$ with $\ell(e)\eq i$. 
Set $$u_{\se_i}(y)= 1.$$
For $e\!\ne\!\se_i$, the construction is subdivided into two cases.

\textbf{Case A}.
If $\lbrp{i,\ell(v_e^+)}\!\not\subset\!I_+$,
then
$$
 e\in\bfE(\lt_{(I)})\!\sqcup\!\bbI_\cht(\lt_{(I)})
 ~\big(=\!\wh\Edg(\lt_{(I)})\bsl
 \{\se_i\!:i\inn\bbI_+(\lt_{(I)})\}\big),
$$
hence $\mu_e$ exists, and $\mu_e\eq 0$ if and only if $e\inn\bbI_\cht(\lt_{(I)})$.
In Case A,
since $\lbrp{i,\ell(v_e^+)}\!\not\subset\!I_+$,
(\ref{Eqn:t_(I)_bbI}) further implies
\begin{equation}\label{Eqn:u_e_induc_1'}
\mu_e\eq 0\qquad
\Longleftrightarrow\qquad
 e\inn\bbI_\cht\bsl I_\cht.
\end{equation}
Let
$$
 \ka=\ka_e:=\min\big(\lbrp{i,\ell(v_e^+)}\bsl I_+\big)\quad \big(\in \bbI_+(\lt_{(I)})\big).
$$
Since $\lbrp{i,\ka}\!\subset\!I_+$,
we have 
$$
 \ell_{(I)}(e)=\ka.
$$
Let $u_e(y)$ be given by
\begin{equation}\label{Eqn:u_e_induc_1}
 \mu_{e}\eq 
 u_e(y)
 \!\cdot\!
 \frac{u_{\se_i^+}\!(y)\!\cdot\! u_{\se_{i[1]}^+}\!\!(y)\cdots}
 {u_{\se_{\ka}^+}\!(y)\!\cdot\! u_{\se_{\ka[1]}^+}\!\!(y)\cdots}
 \!\cdot\! 
 \frac{\prod_{\fe\succ \se_{\ka},\,\lbrp{\ell(\fe),\ell(v_\fe^+)}_\lt\subset I}\ze_{\fe}(\fz)}
 {\prod_{e'\succ e,\,\lbrp{\ell(e'),\ell(v_{e'}^+)}_\lt\subset I}\ze_{e'}\!(\fz)}\!\cdot\!\!\!
 \prod_{
 h\in\lbrp{i,\ka}_\lt}\!\!\!\!
 \ve_h(y).
\end{equation}
We observe that if $i\!\not\in\! I_+$,
then $\ka\eq i$ and hence $\prod_{
 h\in\lbrp{i,\ka}_\lt}
 \ve_h(y)\eq 1$;
if $i\inn I_+$, then the previous construction of $\ve_i(y)$, along with the inductive assumption,
guarantees $\prod_{
 h\in\lbrp{i,\ka}_\lt}
 \ve_h(y)\!\ne\!0$.
Mimicking the argument of Case 2 of the construction of $\ve_i(y)$ and taking~(\ref{Eqn:u_e_induc_1'}) into account,
we see that $u_e(y)$ determined by~(\ref{Eqn:u_e_induc_1}) is a rational function in $\ze_{e'}(\fz)$ and $\mu_{e''}$, and it satisfies~(\ref{Eqn:ve_u_e_inductive}).

\textbf{Case B}.
If $\lbrp{i,\ell(v_e^+)}\!\subset\!I_+$,
then~(\ref{Eqn:edges_cntrd}) gives
\begin{equation}
\label{Eqn:u_e_induc_2'}
e\in\Edg(\lt_{(I)})\ \ \big(\tn{i.e.}~\ze_e(\fz)\eq 0\big)\quad
\Longleftrightarrow\quad
e\in\bbI_\cht\bsl I_\cht.
\end{equation}
Let $u_e(y)$ be given by
\begin{equation}\begin{split}
\label{Eqn:u_e_induc_2}
 \ze_e(\fz)
 =
 u_e(y)
 \cdot
 \frac{u_{\se_i^+}\!(y)\!\cdot\! u_{\se_{i[1]}^+}\!\!(y)\cdots}
 {u_{e_{v_e^+}}\!\!(y)\!\cdot\!
 u_{\se_{\ell(v_e^+)}^+}\!(y)\!\cdot\!
 u_{\se_{\ell(v_e^+)[1]}^+}\!\!(y)\cdots}
 \cdot\!\!
 \prod_{
 h\in\lbrp{i,\ell(v_e^+)}_\lt}\!\!\!\!
 \ve_h(y).
\end{split}\end{equation}
Once again,
mimicking the argument of Case A of the construction of $u_e(y)$, and taking~(\ref{Eqn:u_e_induc_2'}) as well as the description of $\cV_{\tr_{(I)}}$ right before~(\ref{Eqn:V_(I)}) into account,
we see that $u_e(y)$ determined by~(\ref{Eqn:u_e_induc_2}) is a rational function in $\ze_{e'}(\fz)$ and $\mu_{e''}$, and it satisfies~(\ref{Eqn:ve_u_e_inductive}).

The cases 1-3, A, and B together complete the inductive construction of~$\Psi_{x;(I)}$.
Moreover,
comparing
\begin{enumerate}
[leftmargin=*,label=$\bullet$]
\item (\ref{Eqn:y_other'}) with the second line of~(\ref{Eqn:theta_x}),

\item (\ref{Eqn:v_e_base}), the second case of~(\ref{Eqn:u_e_base}), (\ref{Eqn:ve_i_induc_2}), and~(\ref{Eqn:u_e_induc_2}) with the first line ~(\ref{Eqn:theta_x}), 

\item the first case of~(\ref{Eqn:u_e_base}),
(\ref{Eqn:ve_i_induc_1}), and~(\ref{Eqn:u_e_induc_1}) with the expressions of $\mu_{e;i;I}$ right after~(\ref{Eqn:Phi_x_(I)}),
\end{enumerate}
we observe that $\Psi_{x;(I)}$ is the inverse of $\Phi_{x;(I)}$.
\end{proof}

\begin{crl}
$\Phi_x\!:\fU_x\!\lra\!\mtd{}$ is injective.
\end{crl}

\begin{proof}
This follows from Lemma~\ref{Lm:Phi_x_inj} and the stratification~(\ref{Eqn:ti_V_strata}) and~(\ref{Eqn:MtdStrata})
directly.
\end{proof}

\subsection{Stack structure}
\label{Subsec:transition_maps}

In this subsection,
we will show the twisted charts $\Phi_x$ patch together to endow $\mtd{}$ with a smooth stack structure;
c.f.~Proposition~\ref{Prp:Chart_compatible} and Corollary~\ref{Crl:MtdSmooth}.
Note that a priori, $\Phi_{x}$ depends on the choices of the special vertices $\{\sv_i\}_{i\in\bbI_+}$~(\ref{Eqn:sv_i}) and of the local parameters~(\ref{Eqn:local_parameters}). 
Nonetheless,
Lemmas~\ref{Lm:Chart_indep_svi} and~\ref{Lm:Chart_compatible_ptws} below will guarantee that such choices do not affect the proposed stack structure of $\mtd{}$,
hence will make the proof of Proposition~\ref{Prp:Chart_compatible} more concise.

Let $\{\sw_i\!:i\inn\bbI_+\}$ be another set of the special vertices satisfying~(\ref{Eqn:sv_i}),
and $\sw_i^+$,~$\sd_i$, and $\sd_i^+$ be the analogues of $\sv_i^+$, $\se_i$, and $\se_i^+$ in~(\ref{Eqn:se_i}), respectively.
As in~(\ref{Eqn:i[h]}),
each level $i\inn\bbI_+$ similarly determines a {\it finite} sequence
$$
 i\lr 0=i<
 i\lr 1=\ell(\sw_i^+)<
 i\lr 2=\ell(\sw_{i\lr 1}^+)<\cdots.
$$
We take an open subset
$$
 \fU_x^\sa\subset\A^{\bbI_+}\times
 \A^{\wh\Edg(\lt)\bsl\{\sd_i:i\in\bbI_+\}}\times
 \A^{\bbI_-}\times\A^J
$$
with the coordinates 
$$
 \big(
(\de_i)_{i\in\bbI_+},
(u^\sa_e)_{e\in\wh\Edg(\lt)\bsl\{\sd_i:{i\in\bbI_+}\}},
(z^\sa_e)_{e\in\bbI_-},
(w^\sa_j)_{j\in J}
\big).
$$
as in~(\ref{Eqn:ti_cV}),
and then construct
$$
\th^\sa_x\!:\fU_x^\sa\lra\cV,\qquad
\mu^\sa_{e;i;I}\inn\Ga\big(\sO_{\fU_{x;(I)}^{\sa,\circ}}\big),\qquad
\Phi^\sa_x\!:\fU_x^\sa\lra\mtd{}
$$
parallel to~(\ref{Eqn:theta_x}) and~(\ref{Eqn:Phi_x_(I)}).
Let $\cU\eq\Phi_x^\sa(\fU_x^\sa)\!\cap\!\Phi_x(\fU_x)$.

\begin{lmm}
\label{Lm:Chart_indep_svi}
The transition map
$$
 (\Phi_x^\sa)^{-1}\!\circ\!\Phi_x:\,
\Phi_x^{-1}(\cU)\lra(\Phi_x^\sa)^{-1}(\cU)
$$ is an isomorphism.
\end{lmm}

\begin{proof}
Let $g\!:\Phi_x^{-1}(\cU)\!\lra\!(\Phi_x^\sa)^{-1}(\cU)$ be the isomorphism  given by
\begin{gather*}
 g^*\de_i
 \eq
 \ve_i\!\cdot\! 
 \frac{u_{\se_i^+}\!\cdot\! u_{\se_{i[1]}^+}\!\cdots
 }{
 u_{\se_{i^\sharp}^+}\!\cdot\! u_{\se_{i^\sharp[1]}^+}\!\cdots}
 \frac{ u_{\sd_i}\!\cdot\! u_{\sd_{i\lr 1}}\!\cdots
 }{
 u_{\sd_{i^\sharp}}\!\cdot\!u_{\sd_{i^\sharp\lr 1}}\!\cdots}
 \frac{ u_{\sd_{i^\sharp}^+}\!\cdot\! u_{\sd_{i^\sharp\lr{1}}^+}\!\cdots}{ u_{\sd_i^+}\!\cdot\! u_{\sd_{i\lr 1}^+}\!\cdots}
 \qquad\forall\,i\inn\bbI_+,
 \\
 g^*u^\sa_e\eq \frac{u_e}{ u_{\sd_{\ell(e)}}}\ \forall\,e\inn
 \wh\Edg(\lt)\bsl\{\sd_i\}_{i\in\bbI_+}, \quad
 g^*z^\sa_e\eq z_e\ \forall\,e\inn\bbI_-,\quad
 g^*w_j^\sa\eq w_j\ \forall\,j\inn J.
\end{gather*}

The fact that $g$ is an isomorphism
can be shown by
constructing its inverse explicitly,
which is similar to the proof of Lemma~\ref{Lm:Phi_x_inj}, but is simpler.
The key point of the construction is that
$$
u_{\sd_i}=\frac{1}{g^*u^\sa_{\se_i}}\quad
\forall\,i\inn\bbI_+,\qquad
u_e=\frac{g^*u^\sa_e}{g^*u^\sa_{\se_i}}\quad
\forall\,e\inn\wh\Edg(\lt)\ \tn{with}\ \ell(e)\eq i,
$$
and each $\ve_i$ is a product of $g^*\de_i$ and a rational function of $u_e$ with $e\inn\wh\Edg(\lt)$.

It is a direct check that the isomorphism $g$ satisfies
$$
 \th_x^\sa\!\circ\!g\eq\th_x\qquad\tn{and}\qquad
 g^*\mu_{e;i;I}^\sa\eq
 \frac{
 \mu_{e;i;I}}{
 \mu_{\sd_i;i;I}}
 \quad
 \forall\,I\!\subset\!\bbI,~i\inn\bbI_+\bsl I_+,~e\inn\fE_i.
$$
Thus, $(\Phi_x^\sa)^{-1}\!\circ\!\Phi_x\eq g$ and hence is an isomorphism.
\end{proof}

Let $$\{\wh\ze_e\!:e\inn\Edg(\ga)\}\!\sqcup\!\{\wh\vs_j\!:j\inn J\}$$ be another set of extended modular parameters centered at $x$ on the same chart $\cV\!\lra\!\mwt$; see~(\ref{Eqn:local_parameters}).
We use this set of local parameters to construct another twisted chart $\wh\Phi_x\!:\wh\fU_x\!\lra\!\mtd{}$;
in particular,
we have $\wh\th_x\!:\wh\fU_x\!\lra\!\cV$ and $\wh\mu_{e;i;I}\inn\Ga\big(\sO_{\wh\fU_{x;(I)}^\circ}\big)$ as in~(\ref{Eqn:theta_x}) and~(\ref{Eqn:Phi_x_(I)}), respectively.
Parallel to~(\ref{Eqn:ti_cV}), the coordinates on $\wh\fU_x$ are denoted by
$$
 \big(
 (\wh\ve_i)_{i\in\bbI_+},
(\wh u_e)_{e\in\wh\Edg(\lt)\bsl\{\se_i:{i\in\bbI_+}\}},
(\wh z_e)_{e\in\bbI_-},
(\wh w_j)_{j\in J}
 \big).
$$
Let $\cU\eq\Phi_x(\fU_x)\!\cap\!\wh\Phi_x(\wh\fU_x)$.

\begin{lmm}
\label{Lm:Chart_compatible_ptws}
The transition map
$$
 (\wh\Phi_x)^{-1}\!\circ\!\Phi_x:\,
\Phi_x^{-1}(\cU)\lra\wh\Phi_x^{-1}(\cU)
$$ is an isomorphism.
\end{lmm}

\begin{proof}
By Lemma~\ref{Lm:Chart_indep_svi},
it suffices to use the same set of the special vertices $\{\sv_i\}_{i\in\bbI_+}$ for both $\Phi_x$ and $\wh\Phi_x$. For any $e\inn\Edg(\ga)$, 
the local parameters $\wh\ze_e$ and $\ze_e$ defines the same locus $\cZ_e\eq \{\ze_e\eq 0\}\eq \{\wh\ze_e\eq 0\}$,
hence
there exists $f_e\inn\Ga(\sO_\cV^*)$ such that
$$
 \wh\ze_e=f_e\cdot \ze_e.
$$
Therefore, we have
\begin{equation}
\label{Eqn:different_modular_parameters}
 \prt_{\wh\ze_e}|_{\cZ_e}=
 \big(\tfrac{1}{f_e}\prt_{\ze_e}\big)\big|_{\cZ_e}.
\end{equation}

Let $g\!:\Phi_x^{-1}(\cU)\!\lra\!\wh\Phi_x^{-1}(\cU)$ be the isomorphism given by
\begin{gather*}
g^*\wh\ve_i\eq
\ve_i\!\cdot\!\frac{f_{\se_i}\!\cdot\!f_{\se_{i[1]}}\!\cdots
 }{
 f_{\se_{i^\sharp}}\!\cdot\!f_{\se_{i^\sharp[1]}}\!\cdots}~
 \forall\,i\inn\bbI_+,\qquad
g^*\wh z_e\eq z_e~\forall\,e\inn\bbI_-,\qquad
 g^*\wh w_j\eq w_j~\forall\,j\inn J,
\\
g^*\wh u_e\eq 
u_e\!\cdot\!\frac{\prod_{e'\succeq e}f_{e'}
}{
\prod_{e'\succeq \se_{\ell(e)}}f_{e'}}
\quad
\forall\,e\inn\wh\Edg(\lt)\bsl\{\se_i:i\inn\bbI_+\}.
\end{gather*}
The explicit expression of $g$ above implies it is invertible;
see the parallel argument in the proof of Lemma~\ref{Lm:Chart_indep_svi}.

It is a direct check that 
$
 \wh\th_x\!\circ\!g\eq\th_x
$
and 
$$
 \frac{
 g^*\wh\mu_{e;i;I}
 }{
 \prod_{\fe\succeq e,\,
 \lbrp{\ell(\fe),\ell(v_\fe^+)}\not\subset I}
 f_{\fe}}
 =
 \frac{
 \mu_{e;i;I}
 }{
 \prod_{\fe\succeq\se_{i},\,
 \lbrp{\ell(\fe),\ell(v_\fe^+)}\not\subset I}
 f_{\fe}}\quad
 \forall\,I\!\subset\!\bbI,~i\inn\bbI_+\bsl I_+,~e\inn\fE_i.$$
Taking~(\ref{Eqn:different_modular_parameters}) into account,
we conclude that $(\wh\Phi_x)^{-1}\!\circ\!\Phi_x\eq g$ and hence is an isomorphism.
\end{proof}
 
Given $I\!\subset\!\bbI$ and
$x'\inn\Phi_x\big(\fU_{x;[\lt_{(I)}]}\big)$,
let $\cV_{\varpi(x')}\!\lra\!\mwt$ be 
a chart containing $\varpi(x')$ and $\Phi_{x'}\!:\fU_{x'}\!\lra\!\mtd{}$ be a twisted chart centered at $x'$ over $\cV_{\varpi(x')}\!\lra\!\mwt$.
Let $\cU\eq\Phi_x(\fU_x)\!\cap\!\Phi_{x'}(\fU_{x'}).$

\begin{prp}
\label{Prp:Chart_compatible}
The transition map
$$
 \Phi_{x'}^{-1}\circ\Phi_x:\,
 \Phi_x^{-1}(\cU)\lra\Phi_{x'}^{-1}(\cU)
$$ 
is an isomorphism.
\end{prp}

\begin{proof}
Since $x'\inn\Phi_x\big(\fU_{x;\lt_{(I)}}\big)\!\subset\!\Phi_x(\fU_x)$,
its underlying weighted curve satisfies
$$
 \varpi(x')\in\cV_{(I)}^\circ\quad
 \Big(=\big\{\ze_e\!\ne\!0:e\inn\Edg(\lt)\bsl\Edg(\lt_{(I)}\big)\big\}
 \subset \cV\Big).
$$
Thus, replacing $\cV_{\varpi(x')}$ by $\cV_{\varpi(x')}\!\cap\!\cV_{(I)}^\circ$ if necessary,
we may assume 
$$
 \varpi(x')\in
 \cV_{\varpi(x')}
 \subset\cV_{(I)}^\circ.
$$
Moreover,
the following modular parameters on $\cV$:
$$
 \ze_e\,,\quad e\in\Edg\big(\lt_{(I)}\big)
 \subset\Edg(\lt)
$$
also serve as modular parameters on $\cV_{\varpi(x')}$.
Thus by Lemma~\ref{Lm:Chart_compatible_ptws},
we may assume $\Phi_{x'}$ is constructed using the local parameters on $\cV_{\varpi(x')}$:
$$
\{\ze_e\}_{e\in\Edg(\lt_{(I)})}\sqcup
\Big(
\{\vs_j\}_{j\in J}
\sqcup
\{\ze_e\}_{e\in\Edg(\lt)\bsl\Edg(\lt_{(I)})}
\Big)
$$
as the analogue of~(\ref{Eqn:local_parameters}).

Let the special vertices $\sv_i$ and edges $\se_i$ of $\lt$ be respectively as in~(\ref{Eqn:sv_i}) and~(\ref{Eqn:se_i}).
By Lemma~\ref{Lm:Chart_indep_svi},
we may further assume that the special vertices  and edges of $\lt_{(I)}$ are respectively
$$
\sv_i\quad\tn{and}\quad\se_i,\qquad
i\inn\bbI_+\big(\lt_{(I)}\big)\eq\bbI_+\bsl I_+.
$$
For any $i\inn\bbI_+\bsl I_+$ and $h\inn\Z_{\ge 0}$,
let $i^\uparrow$, $\se_i^\dag$, and $i(h)$ be the analogues of $i^\sharp$, $\se_i^+$, and $i[h]$,
respectively, for the weighted level tree $\lt_{(I)}$ instead of $\lt$;
see~(\ref{Eqn:i^sharp}), (\ref{Eqn:se_i}), and~(\ref{Eqn:i[h]}) for notation.

Recall that
$
\bbI_-\big(\lt_{(I)}\big)\eq 
\bbI_-\bsl I_-\!\sqcup\!
\big\{e\inn\bbI_\cht\bsl I_\cht\!:
\ell(v_e^+)\!\le\!\cht\big(\lt_{(I)}\big)\big\}.
$
We denote by
$$
\big(
(\ve_i')_{i\in\bbI_+\!\bsl I_+},\,
(u_e')_{e\in\wh\Edg(\lt_{(I)})\bsl\{\se_i:{i\in\bbI_+\!\bsl I_+}\}},\,
(z_e')_{e\in\bbI_-(\lt_{(I)})},\,
(w_j')_{j\in J\sqcup(\Edg(\lt)\bsl\Edg(\lt_{(I)}))}
\big)
$$
the coordinates on $\fU_{x'}$ 
parallel to~(\ref{Eqn:ti_cV}),
and construct
$$
\th_{x'}\!:\fU_{x'}\!\lra\!\cV_{\varpi(x')}
\quad\tn{and}\quad
\mu'_{e;i;I'}\inn\Ga\big(\sO_{\fU_{x';(I')}^{\circ}}\big),\ 
I'\!\subset\!\bbI\bsl I,\,i\inn\bbI_+\!\bsl(I_+\!\sqcup\!I'),\,e\inn\fE_i
$$
parallel to
$$
 \th_x:\fU_x\lra\cV\qquad
 \tn{and}\qquad
 \mu_{e;i;I}\in\Ga\big(\sO_{\fU_{x;(I)}^\circ}\big),
 \quad i\inn\bbI_+\!\bsl I_+,~e\inn\fE_i,
$$
of~(\ref{Eqn:theta_x}) and~(\ref{Eqn:Phi_x_(I)}), respectively.
In this way,
$
 \Phi_{x'}\!:\fU_{x'}\!\lra\!\mtd{}
$ 
is constructed analogously to $\Phi_x$.

Let $g\!:\Phi_x^{-1}(\cU)\!\lra\!\Phi_{x'}^{-1}(\cU)$ be the isomorphism given by
\begin{equation*}\begin{split}
g^*\ve_i'
=\ve_i 
&\cdot
\Big(\!\!
\prod_{h\in\lprp{i,i^\uparrow}_\lt}\!\!\!\!\!
\ve_h\Big)
\cdot
\frac{
\prod_{\fe\succ_\lt\se_{i^\uparrow},\,
\lbrp{\ell(\fe),\ell(v_\fe^+)}_\lt\subset I}\th_x^*\ze_\fe
}{
\prod_{\fe\succ_\lt\se_i,\,
\lbrp{\ell(\fe),\ell(v_\fe^+)}_\lt\subset I}\th_x^*\ze_\fe}
\cdot
\frac{
u_{\se_i^+}\!\cdot\!u_{\se_{i[1]}^+}\!\cdots
}{
u_{\se_{i^\uparrow}^+}\!\cdot\!u_{\se_{i^\uparrow[1]}^+}\!\cdots
}
\cdot
\\
&\cdot
\frac{ 
\mu_{\se_{i^\uparrow}^\dag;\,i^\uparrow(1);\,I}
\cdot
\mu_{\se_{i^\uparrow\!(1)}^\dag;\,i^\uparrow(2);\,I}
\cdots
}{
\mu_{\se_{i}^\dag;\,i(1);\,I}
\cdot
\mu_{\se_{i(1)}^\dag;\,i(2);\,I}
\cdots}
\hspace{1in}
\forall\,i\inn\bbI_+\bsl I_+
\end{split}
\end{equation*}
and
\begin{equation*}\begin{split}
 &g^*u_e'\eq\mu_{e;\,\ell_{(I)}\!(e);\,I}
 \ \ \forall\,e\inn\wh\Edg\big(\lt_{(I)}\big)\bsl\{\se_i\!:i\inn\bbI_+
 \!\bsl I_+\},
 \\
 &g^*z_e'\eq u_e\ \ \forall\,e\inn\big\{e\inn\bbI_\cht\bsl I_\cht\!:
\ell(v_e^+)\!\le\!\cht\big(\lt_{(I)}\big)\big\},
 \qquad
 g^*z_e'\eq z_e\ \ \forall\,e\inn\bbI_-\bsl I_-,
 \\
 &g^* w_j'\eq w_j\ \ \forall\,j\inn J,\qquad
 g^* w_e'\eq \th_x^*\ze_e\ \ \forall\,e\inn
 \Edg(\lt)\bsl\Edg(\lt_{(I)}).
\end{split}\end{equation*}
To see $g$ is well defined,
notice that $\cV_{\varpi(x')}\!\subset\!\cV_{(I)}^\circ$ implies that
\begin{equation}\label{Eqn:U_pullback}
\Phi_x^{-1}(\cU)\subset\fU_{x;(I)}^\circ.
\end{equation}
Thus, every $\mu_{e;i;I}$ above can be considered as a function on $\Phi_x^{-1}(\cU)$.
By~(\ref{Eqn:U_pullback}) and~(\ref{Eqn:th_{x;(I)}}), 
the function $\th_x^*\ze_\fe$ with $\fe\inn\wh\Edg(\lt)\bsl\bbI_\cht$ is nowhere vanishing on $\Phi_x^{-1}(\cU)$ whenever $\lbrp{\ell(\fe),\ell(v_\fe^+)}_\lt\!\subset\!I$.
Taking~(\ref{Eqn:al_e_nonzero}) and~(\ref{Eqn:mu_vanishing}) into account,
we conclude that $g$ is well defined.

Once again, the explicit expression of $g$ implies it is invertible;
see the parallel argument in the proof of Lemma~\ref{Lm:Chart_indep_svi}.
Moreover,
it is a direct check that $\th_{x'}\!\circ\!g\eq \th_x$ and
$$
g^*\mu'_{e;i;I'}=\mu_{e;i;I\sqcup I'}
\in\Ga\big(
\sO_{\fU^\circ_{x;(I\sqcup I')}\!\cap\Phi_x^{-1}(\cU)}
\big)
\quad
\forall\,I'\!\subset\!\bbI\bsl I,\,i\inn\bbI_+\!\bsl(I_+\!\sqcup\!I'),\,e\inn\fE_i.
$$
Thus,
$\Phi_{x'}^{-1}\!\circ\!\Phi_x\eq g$ and hence is an isomorphism.
\end{proof}

\begin{crl}
\label{Crl:MtdSmooth}
$\mtd{}$ is a smooth Artin stack that is birational to $\mwt$, with $\{\Phi_x\!:\fU_x\!\lra\!\mtd{}\}_{x\in\mtd{}}$ as smooth charts.
Moreover, the structure of the stratification~(\ref{Eqn:MtdStrata}) is locally identical to the one
induced by \eqref{Eqn:ti_V_strata}.
Furthermore, for any $x\inn\mtd{}$, any chart $\cV\!\lra\!\mwt$ containing $\varpi(x)\inn\mwt$, and any twisted chart $\Phi_x\!:\fU_x\!\lra\!\mtd{}$ centered at $x$ lying over $\cV\!\lra\!\mwt$,
we have
\begin{center}
\begin{tikzpicture}{h}
\draw (0,1.2) node {$\fU_x$}
      (2,1.2) node {$\mtd{}$}
      (0,0) node {$\cV$}
      (2,0) node {$\mwt$}
      (1,1.2) node[below] {\small{$\Phi_x$}}
      (2.2,.6) node {\small{$\varpi$}}
      (-.3,.55) node {\small{$\th_x$}};
\draw[->,>=stealth'] (.25,1.2)--(1.55,1.2);
\draw[->,>=stealth'] (.25,0)--(1.5,0);
\draw[->,>=stealth'] (0,.9)--(0,.2);
\draw[->,>=stealth'] (1.9,.95)--(1.9,.2);
\end{tikzpicture}
\end{center}
where $\varpi$ is the forgetful morphism as in~(\ref{Eqn:MtdStrata}) and $\th_x$ is as in~(\ref{Eqn:theta_x}).
\end{crl}

\begin{proof}
The first statement follows from Proposition~\ref{Prp:Chart_compatible}, 
(\ref{Eqn:Phi_onto}), and the fact that $\varpi$
restricts to the identity map on the preimage of
the open subset
$$
 \big\{(C,\bfw)\inn\mwt:
 \bfw(C_o)\!>\!0\big\}
 \subset\mwt.
$$
Lemma~\ref{Lm:Phi_x_inj} then implies for every $[\lt]\inn\sT_\nL^\wt$,
the stack structure of $\mtd{[\lt]}$ is the same as that induced from the inclusion $\mtd{[\lt]}\!\hookrightarrow\!\mtd{}$;
i.e.~the second statement of Corollary~\ref{Crl:MtdSmooth} holds.
The last statement follows from~(\ref{Eqn:Phi_x_(I)}).
\end{proof}

\begin{rmk}\label{Rmk:smooth}
By~(\ref{Eqn:theta_x}) and Corollary~\ref{Crl:MtdSmooth},
one sees that on an arbitrary twisted chart $\fU_x$ of $\mtd{}$, 
\begin{equation*}\begin{split}
\varpi^*\big(\!\!\prod_{e'\succeq \se_\cht}\!\!\ze_{e'}\big)
&=(u_{\se_\cht^+}u_{\se_{\cht[1]}^+}\cdots)
\prod_{i\in\bbI_\cht}\ve_i,\\
\varpi^*\big(\prod_{e'\succeq e}\ze_{e'}\big)
&=(u_e u_{\se_\cht^+}u_{\se_{\cht[1]}^+}\cdots)
\prod_{i\in\bbI_\cht}\ve_i=
u_e\cdot \varpi^*\big(\!\!\prod_{e'\succeq \se_\cht}\!\!\ze_{e'}\big)
\quad\forall\,e\inn\fE_\cht.
\end{split}\end{equation*}
This, along with the local equations of $\ov M_1(\P^n,d)$ in~\cite[\S5.2]{HL10},
 implies that the primary component of $\ti M^\tf_1(\P^n,d)$ is smooth and $\ti M^\tf_1(\P^n,d)$  contains at worst normal crossing singularities. This observation should be useful for the cases
 of higher genera.
\end{rmk}

\subsection{A simple example}
\label{Subsec:Eg}
Let $[\lt]\eq[\ga,\bfw,\ell]\inn\sT_\nL^\wt$ be
given by the leftmost diagram in Figure~\ref{Fig:Example}.
Then,
$$
\Edg(\lt)\eq\{a,b,c,d\},\qquad
\cht\eq -2,\qquad
\bbI\eq\bbI_+\eq\{-1,-2\}.
$$
Each of the four distinct subsets $I$ of $\bbI$ determines a weighted level tree $\lt_{(I)}$; see Figure~\ref{Fig:Example}.
Let $x\inn\mtd{[\lt]}$ be a weighted curve of genus 1 with twisted fields over $(C,\bfw)\inn\mwt$.
The core and the nodes of $C$ are
labeled by $o$ and by $a,b,c,d$, respectively.

Let $\cV\!\lra\!\mwt$ be an affine smooth chart containing $(C,\bfw)$,
with a set of local parameters 
$$
\{\ze_a, \ze_b, \ze_c, \ze_d\}\sqcup
\{\vs_j\}_{j\in J}
$$
centered at $(C,\bfw)$,
where $\ze_a,\ldots,\ze_d$ are the modular parameters.
There then exist non-zero $\la_c$ and $\la_d$ such that
$$
x=\big(\,\ud 0\;;\; 
[\,0\,,\prt_{\ze_b}|_{\ud 0}\,]\;,\; 
[\,\prt_{\ze_a}|_{\ud 0}\,,\,
\la_c\!\cdot\!(\prt_{\ze_b}\!\otimes\!\prt_{\ze_c})|_{\ud 0}\,, \,
\la_d\!\cdot\!(\prt_{\ze_b}\!\otimes\!\prt_{\ze_d})|_{\ud 0}\,]\;
\big).
$$

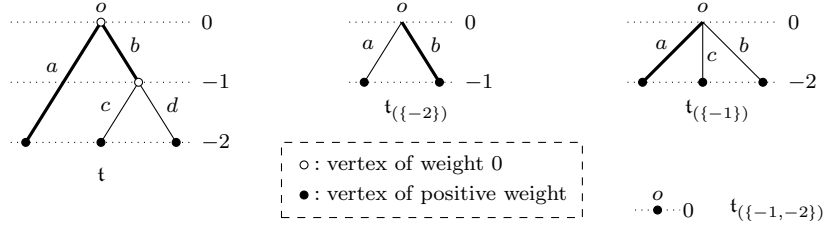
\begin{figure}
\begin{center}
\begin{tikzpicture}{htb}
\draw[dotted]
 (1.8,0)--(4.2,0)
 (1.8,-.8)--(4.2,-.8)
 (1.8,-1.6)--(4.2,-1.6);
\draw
 (3,0)--(2,-1.6)
 (3,0)--(4,-1.6)
 (3.5,-.8)--(3,-1.6);
\draw[very thick]
 (3,0)--(2,-1.6)
 (3,0)--(3.5,-.8);
\draw[fill=white]
 (3,0) circle (1.5pt)
 (3.5,-.8) circle (1.5pt);
\filldraw
 (2,-1.6) circle (1.5pt)
 (4,-1.6) circle (1.5pt)
 (3,-1.6) circle (1.5pt);
\draw
 (4.2,0) node[right] {\scriptsize{$0$}}
 (4.2,-.8) node[right] {\scriptsize{$-1$}}
 (4.2,-1.6) node[right] {\scriptsize{$-2$}}
 (2.35,-0.6) node {\scriptsize{$a$}}
 (3.45,-.3) node {\scriptsize{$b$}}
 (3.06,-1.12) node {\scriptsize{$c$}}
 (3.95,-1.1) node {\scriptsize{$d$}}
 (3,0) node[above] {\scriptsize{$o$}}
 (3,-1.8) node[below] {\scriptsize{$\lt$}};

\draw[dotted, xshift=4cm]
 (2.3,0)--(3.7,0)
 (2.3,-.8)--(3.7,-.8);
\draw[xshift=4cm]
 (3,0)--(2.5,-.8)
 (3,0)--(3.5,-.8);
\draw[very thick,xshift=4cm]
 (3,0)--(3.5,-.8);
\draw[fill=black,xshift=4cm]
 (2.5,-.8) circle (1.5pt)
 (3.5,-.8) circle (1.5pt);
\draw[xshift=4cm]
 (3.7,0) node[right] {\scriptsize{$0$}}
 (3.7,-.8) node[right] {\scriptsize{$-1$}}
 (2.55,-0.3) node {\scriptsize{$a$}}
 (3.45,-.3) node {\scriptsize{$b$}}
 (3,0) node[above] {\scriptsize{$o$}}
 (3.2,-.9) node[below] {\scriptsize{$\lt_{(\{-2\})}$}};
 
\draw[dotted, xshift=8cm]
 (2,0)--(4,0)
 (2,-.8)--(4,-.8);
\draw[xshift=8cm]
 (3,0)--(2.2,-.8)
 (3,0)--(3,-.8)
 (3,0)--(3.8,-.8);
\draw[very thick,xshift=8cm]
 (3,0)--(2.2,-.8);
\draw[fill=black,xshift=8cm]
 (2.2,-.8) circle (1.5pt)
 (3,-.8) circle (1.5pt)
 (3.8,-.8) circle (1.5pt);
\draw[xshift=8cm]
 (4,0) node[right] {\scriptsize{$0$}}
 (4,-.8) node[right] {\scriptsize{$-2$}}
 (2.45,-0.3) node {\scriptsize{$a$}}
 (3.55,-.3) node {\scriptsize{$b$}}
 (3.12,-.45) node {\scriptsize{$c$}}
 (3,0) node[above] {\scriptsize{$o$}}
 (3.2,-.9) node[below] {\scriptsize{$\lt_{(\{-1\})}$}};
 
\draw[dotted, xshift=7.4cm, yshift=-2.5cm]
 (2.7,0)--(3.3,0);
\draw[fill=black,xshift=7.4cm,yshift=-2.5cm]
 (3,0) circle (1.5pt);
\draw[xshift=7.4cm,yshift=-2.5cm]
 (3.2,0) node[right] {\scriptsize{$0$}}
 (3,0) node[above] {\scriptsize{$o$}}
 (4.6,0) node {\scriptsize{$\lt_{(\{-1,-2\})}$}};

\draw[xshift=1.7cm,yshift=.1cm]
 (4,-2) circle (1.5pt);
\filldraw[xshift=1.7cm,yshift=.1cm]
 (4,-2.4) circle (1.5pt);
\draw[xshift=1.7cm,yshift=.1cm]
 (4,-2) node[right] {\scriptsize{$:\tn{vertex~of~weight}~0$}}
 (4,-2.4) node[right] {\scriptsize{$:\tn{vertex~of~positive~weight}$}};
\draw[dashed,xshift=1.7cm,yshift=.1cm]
 (3.7,-1.7) rectangle (7.7,-2.7);
\end{tikzpicture}
\end{center}
\caption{Relevant weighted level trees in \S\ref{Subsec:Eg}}\label{Fig:Example}
\end{figure}

We choose the special edges (\ref{Eqn:se_i}) of $\lt$ to be $\se_1\eq b$ and $\se_2\eq a$.
Let
$$
 \fU_x\subset\A^{\{-1,-2\}}\times\A^{\{c,d\}}\times\A^J
$$
be an open subset containing the point 
$$
 y_x= (0,0,\la_c,\la_d,0,\ldots,0).
$$
The coordinates of $\fU_x$ are denoted by 
$$
 \ve_{-1},\,\ve_{-2},\,
 u_c,\, u_d,\quad\tn{and}\quad
 w_j~\tn{with}~j\inn J.
$$
We may take $\fU_x\eq\{u_c\!\ne\!0,u_d\!\ne\!0\}$.

By Corollary~\ref{Crl:MtdSmooth} and~(\ref{Eqn:theta_x}),
the forgetful morphism $\varpi\!:\mtd{}\!\lra\!\mwt$ can locally be written as $\th_x\!:\fU_x\!\lra\!\cV$ such that
$$
 \th_x
 \big(\,
 \ve_{-1},\,\ve_{-2},\,
 u_c,\, u_d,\, (w_j)\,
 \big)
 \eq 
 \big(
 \underbrace{
 \ve_{-1}\!\cdot\!\ve_{-2}}_{\ze_a},
 \underbrace{\ve_{-1}}_{\ze_b},
 \underbrace{
 \ve_{-2}\!\cdot\!u_c}_{\ze_c},
 \underbrace{
 \ve_{-2}\!\cdot\!u_d}_{\ze_d},
 (w_j)\,
 \big).
$$
Considering all possible subsets $I$ of $\bbI$ in (\ref{Eqn:Phi_x_(I)}),
we obtain a twisted chart $\Phi_x\!:\fU_x\!\lra\!\mtd{}$ centered at $x$ over  $\cV\!\lra\!\mwt$
so that for any 
$$
 y=\big(\,
 \ve_{-1},\,\ve_{-2},\,
 u_c,\, u_d,\, (w_j)\,
 \big)\in\fU_x,
$$
\begin{enumerate}
[leftmargin=*,label=$\bullet$]
\item 
if $\ve_{-1}\eq \ve_{-2}\eq 0$,
then 
\begin{equation*}\begin{split}
 \Phi_x(y)\eq
 \Big(
  \big(&0,0,0,0,(w_j)\big);\\
  [&0,\prt_{\ze_b}\!|_{\th_x(y)}],\;
  [\prt_{\ze_a}\!|_{\th_x(y)},\,u_c
  (\prt_{\ze_b}\!\otimes\!\prt_{\ze_c})\!|_{\th_x(y)},\,
  u_d
  (\prt_{\ze_b}\!\otimes\!\prt_{\ze_d})\!|_{\th_x(y)}]
  \Big)
  \in\mtd{[\lt]};
\end{split}\end{equation*}

\item 
if $\ve_{-1}\!\ne\! 0$ and $\ve_{-2}\eq 0$,
then
$$
\Phi_x(y)\eq 
\Big(\!
   \big(0,\ve_{-1},0,0,\,(w_j)\big);\,
   [\prt_{\ze_a}\!|_{\th_x(y)}\,,\,
 \tfrac{u_c}{\ve_{-1}}\!\cdot\!
   \prt_{\ze_c}\!|_{\th_x(y)}\,,\,
 \tfrac{u_d}{\ve_{-1}}\!\cdot\!
 \prt_{\ze_d}\!|_{\th_x(y)}]
   \Big)
  \in\mtd{[\lt_{(\{-1\})}]};
$$

\item 
if $\ve_{-1}\eq 0$ and $\ve_{-2}\!\ne\! 0$,
then
$$
\Phi_x(y)\eq 
\Big(
   \big(0,\,0,\,\ve_{-2}\!\cdot\!u_c,\,\ve_{-2}\!\cdot\!u_d,\,(w_j)\big);\,
   [\ve_{-2}\!\cdot\!\prt_{\ze_a}\!|_{\th_x(y)}\,,\,
   \prt_{\ze_b}\!|_{\th_x(y)}]
   \Big)
  \in\mtd{[\lt_{(\{-2\})}]};
$$
  
\item 
if $\ve_{-1}\!\ne\! 0$ and $\ve_{-2}\!\ne\! 0$,
then
$$
\Phi_x(y)=\th_x(y)=
   \big(\ve_{-1}\!\cdot\!\ve_{-2},\,
   \ve_{-1},\,\ve_{-2}\!\cdot\!u_c,\,\ve_{-2}\!\cdot\!u_d,(w_j)\big)
  \in\mtd{[\lt_{(\{-1,-2\})}]}.
$$
\end{enumerate}

With the expressions of $\Phi_x$ as above,
it is straightforward to check 
$$
 \Phi_x(0,0,\la_c,\la_d,0,\ldots,0)=x,
$$
as well as to verify the statements of Lemmas~\ref{Lm:Phi_x_inj}, \ref{Lm:Chart_indep_svi}, \ref{Lm:Chart_compatible_ptws},
and Proposition~\ref{Prp:Chart_compatible} in this situation.

\subsection{Universal family}\label{Subsec:Universal_family}
Let $\pi^\wt\!:\cC^\wt\!\lra\! \mwt$ be the universal weighted  nodal curves of genus 1.
The stratification~(\ref{Eqn:Mwt_strata}) gives rise to a stratification
$$
 \cC^\wt=\!\bigsqcup_{\tr\in\sT_\nR^\wt}\!\!\cC^\wt_\tr\qquad
 \tn{satisfying}\qquad
 \pi^\wt(\cC^\wt_\tr)\eq\mwt_\tr\quad
 \forall\,\tr\inn\sT_\nR^\wt.
$$
Parallel to~(\ref{Eqn:MtdStrat_and_sEcT}) and~(\ref{Eqn:MtdStrata}),
we set
\begin{equation*}
\begin{split}
 \cC^\tf_{[\lt]}&=
 \bigg(
 \prod_{i\in\bbI_+(\lt)}\!\!
 \Big\lgroup\!
  \Big(\,\mathring\P\big(\!
   \bigoplus_{\!
   \begin{subarray}{c}
    e\in\Edg(\lt),\,
    \ell(v_e^-)=i
   \end{subarray}
   }\!\!\!\!\!\!\!\!\!\!\!\!
   (\pi^\wt)^*L_e^\succeq\;
  \big)\!\Big)\Big/
  \cC^\wt_{\ff[\lt]}\Big\rgroup\!\!
 \bigg)
 \stackrel{}{\lra}\cC^\wt_{\ff[\lt]}\qquad
 \forall\,[\lt]\inn\sT_\nL^\wt, 
 \\
 \cC^\tf&=
 \!\bigsqcup_{[\lt]\in\sT_\nL^\wt}\!\!\!
 \cC^\tf_{[\lt]}
 \stackrel{}{\lra}\cC^\wt.
\end{split}
\end{equation*} 
Mimicking the construction of the stack structure of $\mtd{}$ in \S\ref{Subsec:Charts_of_Mtf} and \S\ref{Subsec:transition_maps},
we can endow $\cC^\tf$ with a stack structure analogously.
Furthermore,
the projection $\pi^\wt\!:\cC^\wt\!\lra\! \mwt$ induces a unique projection
\begin{equation}\label{Eqn:tf_universal_family}
\pi^\tf:\cC^\tf\lra\mtd{}.
\end{equation}
It is straightforward that
$$
\cC^\td\cong\cC^\wt\!\times_{\mwt}\!\mtd{}\lra\mtd{}.
$$

For any scheme $S$, a flat family  $\cZ/S$ of stable weighted nodal curves of genus 1 with twisted fields
corresponds to a morphism 
\hbox{$f\!: S\lra\mtd{}$} such that $\cZ/S$ is the pullback of~(\ref{Eqn:tf_universal_family}):
\begin{equation}\label{Eqn:family}
\cZ/S \cong (S \times_{\mtd{}} \cC^\td)/S\,.
\end{equation}
This leads to the following statement.

\begin{prp}
 \label{Prp:Stack}
  $\cC^\tf\!\lra\!\mtd{}$ in~(\ref{Eqn:tf_universal_family}) gives the universal family of $\mtd{}$. 
\end{prp}

\begin{rmk}\label{Rmk:moduli}
 One may establish a moduli interpretation of $\mtd{}$ as follows:
for any scheme $S$,
every flat family $\cZ/S$ of stable weighted nodal curves of genus 1 with twisted fields can be constructed directly as follows.
A priori,
$\cZ/S$ should be over a flat family $\cC_S/S$ of stable weighted curves,
thus by the universality of the moduli $\mwt$,
there exists a morphism
$$
 \al:S\lra\mwt
 =\!\bigsqcup_{\tr\in\sT_\nR^\wt}\!\!\mwt_\tr
$$
such that $\cC_S/S$ is the pullback of $\cC^\wt\!/\,\mwt$ via $\al$.
This induces a stratification of the scheme~$S$:
\begin{equation}\label{Eqn:S_stratification}
 S=\!\bigsqcup_{\tr\in\sT_\nR^\wt}\!\!S_\tr\qquad
 \tn{satisfying}\qquad
 \al(S_\tr)\!\subset\!\mwt_\tr\quad
 \forall\,\tr\inn\sT_\nR^\wt.
\end{equation} 
We take
\begin{equation*}
\begin{split}
 S^\tf_{[\lt]}&=
 \bigg(
 \prod_{i\in\bbI_+(\lt)}\!\!\!
  \Big\lgroup\!
  \Big(\,\mathring\P
  \big(\!
   \bigoplus_{\!
   \begin{subarray}{c}
    e\in\Edg(\lt),\,
    \ell(v_e^-)=i
   \end{subarray}
   }\!\!\!\!\!\!\!\!\!\!\!
   \al^*L_e^\succeq\;
   \big)
   \Big)\Big/S_{\ff[\lt]}
  \Big\rgroup\!\!
 \bigg)
 \stackrel{\pi_S}{\lra}S_{\ff[\lt]}\qquad
 \forall\,[\lt]\inn\sT_\nL^\wt,\\
 S^\tf&=
 \!\bigsqcup_{[\lt]\in\sT_\nL^\wt}\!\!
 S^\tf_{[\lt]}
 \stackrel{\pi_S}{\lra}S.
\end{split}
\end{equation*} 

For any chart $\cV_S\!\lra\!S$,
shrinking it if necessary,
we see there exists a (smooth) chart $\cV\!\lra\!\mwt$ such that 
\begin{center}
\begin{tikzpicture}{h}
\draw (0,1) node {$\cV_S$}
      (1.9,1) node {$\cV$}
      (0,0) node {$S$}
      (2,0) node {$\mwt$};
\draw[->,>=stealth'] (.25,1)--(1.65,1);
\draw[->,>=stealth'] (.25,0)--(1.5,0);
\draw[->,>=stealth'] (0,.7)--(0,.2);
\draw[->,>=stealth'] (1.9,.75)--(1.9,.2);
\end{tikzpicture}
\end{center}
commutes.
The modular parameters $\ze_e$ on $\cV$ pull back to regular functions on $\cV_S$,
which are denoted by $\ze^S_e$.
By~(\ref{Eqn:S_stratification}),
we have
$$
 S_\tr\cap\cV_S=\{\ze^S_e\eq 0:e\inn\Edg(\tr)\}\cap\{\ze^S_{e'}\!\ne\!0:e'\!\not\in\!\Edg(\tr)\}.
$$
Mimicking the construction in \S\ref{Subsec:Charts_of_Mtf} and \S\ref{Subsec:transition_maps},
we can thus endow $S^\tf$ with a scheme structure.

We say $\cZ/S$ is a flat family of stable weighted nodal curves of genus 1 with twisted fields
if and only if there exists a section $\si^\tf$ of $\pi_S\!:S^\tf\!\lra S$ such that
$$
 \cZ=\cC_S\times_S\big(\si^\tf(S)\big).
$$
This construction is consistent with~(\ref{Eqn:family}). 
One can check that the groupoid sending any scheme $S$ to the set of all such defined flat families $\cZ/S$ is represented by $\mtd{}$.

We would like to remark that a more succinct definition of 
a flat family of stable weighted nodal curves of genus 1 with twisted fields
should be desirable.
\end{rmk}

\section{Comparison with Hu-Li's blowup stack $\ti\fM^{\tn{wt}}$}

Let $\pi\!:\tmwt\!\lra\!\mwt$ be the sequential blowup constructed in~\cite[\S2.2]{HL10}.
Since $\mwt$ is a smooth Artin stack and the blowup centers are all smooth, 
so is $\tmwt$.
As per the convention of this paper, we omit the subscript indicating the genus.
In Proposition~\ref{Prp:Isomorphism},
we show that $\mtd{}$ is isomorphic to $\tmwt$.
Lemma~\ref{Lm:Blowup} is rather technical; it is only used in the proof of Proposition~\ref{Prp:Isomorphism}.

We briefly recall the notion of the \ts{locally tree compatible blowups} described in~\cite[\S3]{HLN}.
Let $\fM$ be a smooth stack, $\ga$ be a rooted tree, and $\cV$ be an affine smooth chart of $\fM$.
If there exists a set of local parameters on $\cV$ so that a subset of which can be written as
\begin{equation*}
 \big\{\,
 z_e\inn\Ga\big(\sO_{{\cV}}\big):
 e\in\Edg(\ga)\,
 \big\}\,,
\end{equation*}
then the set is called a \ts{$\ga$-labeled subset of local parameters} on $\cV$.
For example,
if $\fM\eq\mwt$ and $\cV$ is a chart centered at a weighted curve whose reduced dual tree is $\ga$,
then the set of the modular parameters $\{\ze_e\}_{e\in\Edg(\ga)}$ is a $\ga$-labeled subset of local parameters. 

Let $\Ver(\ga)_{\min}$ be the set of the minimal vertices of $\ga$ with respect to the tree order.
We call a subset $\fS$ of $\Edg(\ga)$ a \ts{traverse section} if for any $v\inn\Ver(\ga)_{\min}$,
the path between $o$ and $v$ contains exactly one  element of~$\fS$.
For example,
the subsets $\fE_i$ of $\Edg(\ga)$ as in~(\ref{Eqn:fE_i}) are traverse sections.
Let $\Xi(\ga)$ be the set of the traverse sections.
The tree order on $\Edg(\ga)$ induces a partial order on $\Xi(\ga)$ such that
$$
 \fS\succ\fS'\quad
 \Longleftrightarrow\quad
 \big\lgroup\,
 \fS\!\ne\!\fS'\,
 \big\rgroup~\tn{and}~
 \big\lgroup\,
 \forall~e\inn\fS,~
 \exists~e'\inn\fS'~\tn{s.t.}~e\!\succeq\!e'\,
 \big\rgroup.
$$
We remark that the tree order on $\Edg(\ga)$ and the induced order on $\Xi(\ga)$ in this paper are both {\it opposite} to those in~\cite{HLN},
in order to be consistent with the order of the levels of the weighted level trees.

Let $\ti\fM\!\lra\!\fM$ be the sequential blowup of $\fM$ successively along the proper transforms of the closed substacks
$
 Z_1, Z_2,\ldots
$ of $\fM$.

\begin{dfn}{\cite[Definitions~3.2.4~\&~3.2.1]{HLN}}\label{Dfn:Local_Tree_Comptbl}
The blowup $\ti\fM\!\lra\!\fM$ above is said to be \ts{locally tree-compatible} if
there exists an \'etale cover $\{\cV\}$ of $\fM$
such that for each $\cV\inn\{\cV\}$,
there exist a rooted tree $\ga$, a partition of $\Xi(\ga)$:
 $$
  \Xi(\ga)
  =\bigsqcup_{k\ge 1}
  \Xi_k(\ga)
 $$
and a $\ga$-labeled subset of local parameters on $\cV$ such that
\begin{itemize}[leftmargin=*]
\item
 for every $k\ge 1$,
 $$
  Z_{k}\cap\cV= 
  \bigcup_{\fS\in\Xi_k(\ga)} \!\!\!\!
  \big\{\,z_e\eq 0:\,
  e\inn \fS\,\big\};
 $$

\item
if $\fS'\inn\Xi_{k'}(\ga)$, $\fS''\inn\Xi_{k''}(\ga)$, and $\fS'\!\succ\!\fS''$,
then $k'\!<\!k''$.
\end{itemize}
\end{dfn}

If a sequential blowup $\ti\fM\!\lra\!\fM$ is locally tree-compatible,
then the blowup procedure is finite on each $\cV\inn\{\cV\}$, because the set $\Xi(\ga)$ is finite.

\begin{lmm}\label{Lm:Blowup}
If the blowup $\ti\fM\!\lra\!\fM$ successively along the proper transforms of the closed substacks $Z_1,Z_2,\ldots$ of $\fM$ is locally tree-compatible,
then the blowup $\ti\fM'\!\lra\!\fM$ successively along the total transforms of 
$$
 Y_1=Z_1,\ \ Y_2=Z_1\!\cup\!Z_2,\ \ Y_3=Z_1\!\cup\!Z_2\!\cup\!Z_3,\ \ \ldots
$$
yields the same space, i.e.~$\ti\fM'\eq\ti\fM$.
\end{lmm}

\begin{proof}
We prove the statement by induction.
For each $h\!\ge\!1$,
we will show that after the $h$-th step,
the blowup stacks $\ti\fM_{(h)}'$ of $\fM$ along the total transforms of $Y_1,\ldots,Y_h$ is the same as the blowup $\ti\fM_{(h)}$ of $\fM$ along the proper transforms of $Z_1,\ldots,Z_h$.

The base case of the induction is trivial.
Suppose the blowup $\ti\fM_{(k)}'\eq\ti\fM_{(k)}$.
We will show that for any $x\inn\fM$ and any lift $\ti x$ of $x$ after the $k$-th step,
the blowup along  the total transform $\ti Y_{k+1}$ of $Y_{k+1}$ has the same effect as that along the proper transform $\wc Z_{k+1}$ of $Z_{k+1}$ near $\ti x$.
Since $x$ and $\ti x$ are arbitrary,
this will establish the $(k\!+\!1)$-th step of the induction.

W.l.o.g.~we may assume $x\inn\bigcap_{i=1}^{k+1}Z_k$
(otherwise we simply omit the loci $Z_i$ not containing $x$ and change the indices of $Z_i$ and $Y_i$ accordingly).
The blowup $\ti\fM\!\lra\!\fM$ is locally tree-compatible,
hence there exist a rooted tree $\ga$, an affine smooth chart $\cV$ containing $x$, and a $\ga$-labeled subset of local parameters $z_e$, $e\inn\Edg(\ga)$ on $\cV$ such that
$$
 x\in\{z_e\eq 0 :e\inn\Edg(\ga)\}.
$$
As shown in~\cite[Lemma~3.3.2]{HLN},
there exist traverse sections $\fS_{(k)}\inn\Xi_k(\ga)$ and $\fS_{(k+1)}\inn\Xi_{k+1}(\ga)$ (c.f.~Definition~\ref{Dfn:Local_Tree_Comptbl}), an affine smooth chart $\ti\cV_{\ti x}$, and a subset of local parameters
$$
 \ti\ve_1,\ldots,\ti\ve_k;\qquad
 \ti z_e\ \ \tn{with}~e\inn\fS_{(k+1)}\bsl\fS_{(k)};\qquad
 \wc z_e\ \ \tn{with}~e\inn\fS_{(k+1)}\!\cap\!\fS_{(k)}
$$
on $\ti\cV_{\ti x}$ so that $\wc Z_{k+1}$ is locally given by
\begin{equation*}\begin{split}
 \wc Z_{k+1}\cap\ti\cV_{\ti x}&= \big\{\ti z_e\eq 0:e\inn\fS_{(k+1)}\bsl\fS_{(k)}, \; \wc z_e\eq 0:e\inn\fS_{(k+1)}\!\cap\!\fS_{(k)}\big\}.
\end{split}\end{equation*}
Moreover, by~\cite[(3.13)]{HLN},
the total transform of each $Z_i$ with $1\!\le\!i\!\le\!k$ is locally given by $\{\ti\ve_i\eq 0\}$.
Thus, $\ti Y_{k+1}$ is locally given by
$$
 \ti Y_{k+1}\cap\ti\cV_{\ti x}
 =\big(\wc Z_{k+1}\cap\ti\cV_{\ti x}\big)
 \cup
 \big\{\prod_{1\le i\le k}\!\!\ti\ve_i= 0\big\}.
$$
That is, on the chart $\ti\cV_{\ti x}$, $\wc Z_{k+1}$ and $\ti Y_{k+1}$ are defined by the ideals 
$$
 \mathscr I_{\wc Z_{k+1}}\!\eq\langle \ti z_e:e\inn\fS_{(k+1)}\bsl\fS_{(k)}, \; \wc z_e :e\inn\fS_{(k+1)}\!\cap\!\fS_{(k)} \rangle\quad\tn{and}\quad
 \mathscr I_{\wc Z_{k+1}}\!\big(\!\prod_{1\le i\le k}\!\!\!\ti\ve_i\big),
$$
respectively.
Therefore,
blowing up along $\wc Z_{k+1}$ has the same effect on $\ti\cV_{\ti x}$ as that along $\ti Y_{k+1}$.
\end{proof}

\begin{prp}
\label{Prp:Isomorphism}
$\mtd{}/\mwt$ is  isomorphic to $\tmwt/\mwt$.
 In particular, $\varpi\!:\mtd{}\!\lra\!\mwt$ is proper.
\end{prp}

\begin{proof}
Our goal is to construct two morphisms $\psi_1$ and $\psi_2$ between $\tmwt$ and $\mtd{}$ so that the following diagram 
\begin{center}
\begin{tikzpicture}{h}
\draw (0,1) node {$\tmwt$}
      (3.2,1) node {$\mtd{}$}
      (1.5,0) node {$\mwt$}
      (1.6,1) node[above]{\small{$\psi_2$}}
      (1.6,1) node[below]{\small{$\psi_1$}}
      (.55,.35) node[left]{\small{$\pi$}}
      (2.45,.35) node[right]{\small{$\varpi$}};
\draw[->,>=stealth'] (.51,1.05)--(2.7,1.05);
\draw[->,>=stealth'] (2.69,.95)--(.5,.95);
\draw[->,>=stealth'] (.1,.7)--(1,.1);
\draw[->,>=stealth'] (2.9,.7)--(2,.1);
\end{tikzpicture}
\end{center}
commutes.
Since $\pi$ and $\varpi$
restrict to the identity map on the preimages of
the open subset
$$
 \big\{(C,\bfw)\inn\mwt:
 \bfw(C_o)\!>\!0\big\}
 \subset\mwt,
$$
respectively,
we see that
$\psi_2\!\circ\!\psi_1$ and $\psi_1\!\circ\!\psi_2$ are the identity maps.
This then implies the former statement of Propoistion~\ref{Prp:Isomorphism}.
The latter statement follows from the former as well as the properness of the blowup $\tmwt\!\lra\!\mwt$.

We first construct $\psi_1$.
For each $k\inn\Z_{>0}$,
let $Z_k\subset\!\mwt$ be the closed locus whose {\it general} point is obtained by attaching $k$ smooth positively-weighted rational curves to the smooth 0-weighted elliptic core at pairwise distinct points. 
By Lemma~\ref{Lm:Blowup},
the blowup $\pi\!:\ti\fM^\wt\!\lra\!\mwt$ successively along the proper transforms of $Z_1, Z_2, \ldots$ can be identified with the blowup of $\mwt$ successively along the total tranforms of 
$$
 Y_1\eq Z_1,\quad Y_2\eq Z_1\!\cup\!Z_2,\quad Y_3\eq Z_1\!\cup\!Z_2\!\cup\!Z_3,\quad\ldots
$$

We observe that for each $k\inn\Z_{>0}$,
the pullback $\varpi^{-1}(Y_k)$ to $\mtd{}$ is a Cartier divisor.
In fact, for any
$[\lt]\eq[\ga,\bfw,\ell]\inn\sT_\nL^\wt$ and $x\inn\mtd{[\lt]}$,
let $\fU\!\lra\!\mtd{}$ be a twisted chart
centered at $x$, lying over a chart $\cV\!\lra\!\mwt$. 
In \cite{HL10},
the blowup $\pi$ locally on $\cV$ is proved to be compatible with the weighted tree $(\ov\ga,\ov\bfw)$ obtained by contracting all the edges $e$ of $\ga$ as long as there exists $v\!\succeq\!v_e^+$ satisfying
$\bfw(v)\!>\!0$.
Let $\{\ze_e\!:e\inn\Edg(\ga)\}$ be a set of modular parameters on~$\cV$ as in~(\ref{Eqn:modular_parameters}) and 
$$
\{\ve_i\}_{i\in\bbI_+}\cup
\{u_e\}_{e\in\wh\Edg(\lt)\bsl\{\se_i:i\in\bbI_+\}}\cup
\{z_e\}_{e\in\bbI_-}
$$ 
be the subset of the parameters~(\ref{Eqn:ti_cV}) on $\fU$.
We claim that
\begin{equation}\label{Eqn:Yk_Cartier_Div}
 \varpi^{-1}(Y_k)\cap
 \fU=
 \big\{\prod_{
 \begin{subarray}{c}
 i\in\bbI_+,\,
 |\fE_i|\le k
 \end{subarray}}
 \!\!\!\!\!\!\!
 \ve_i~\,\eq 0\,
 \big\}.
\end{equation}

To show~(\ref{Eqn:Yk_Cartier_Div}),
we first notice that $\varpi^{-1}(Y_k)\!\cap\!
 \fU\eq
 \varpi^{-1}(Y_k\!\cap\!\cV)$ by Corollary~\ref{Crl:MtdSmooth}.
Every irreducible component of $Y_k\!\cap\!\cV$ can be written in the form
$$
 Y_{k,\fS}\!:=\!
 \{\ze_e\eq 0: e\inn\fS\}\quad
 \tn{with}~\fS\inn\Xi(\ov\ga),~|\fS|\!\le\!k,~
 \fS\!\cap\!\big(\wh\Edg(\lt)\bsl\bbI_{\cht}\big)\!\ne\!\emptyset.
$$
For each irreducible component $Y_{k,\fS}$ of $Y_k\!\cap\!\cV$,
the local expression $\th_x$ of $\varpi$  as in~(\ref{Eqn:theta_x}) implies the pullback $\varpi^{-1}(Y_{k;\fS})$ can be written as
\begin{equation}\begin{split}\label{Eqn:Yk_pullback}
 &\big\{\,
 \prod_{h\in\lbrp{\ell(e),\ell(v_e^+)}}\!\!\!\!\!\!\!\!\!
 \ve_h\;\eq 0:\,e\inn\fS\!\cap\!(\wh\Edg(\lt)\bsl\bbI_\cht),\\
 &\hspace{.5in}
 \big(u_e\cdot\!\!\!\!\!\!\!\!\!\prod_{h\in\lbrp{\ell(e),\ell(v_e^+)}}\!\!\!\!\!\!\!\!\!
 \ve_h~\big)\eq 0:\,e\inn\fS\!\cap\!\bbI_\cht,
 \quad
 z_e\eq 0:\,
 e\inn\fS\!\cap\!\bbI_-\,\big\}.
\end{split}\end{equation}
Since $\fS\!\cap\!\big(\wh\Edg(\lt)\bsl\bbI_{\cht}\big)\!\ne\!\emptyset$ and $|\fS|\!\le\!k$,
we can always find $e\inn\fS\!\cap\!\big(\wh\Edg(\lt)\bsl\bbI_{\cht}\big)$ such that 
$|\fE_{h}|\!\le\!|\fS|\!\le\!k$ for all $h\inn\lbrp{\ell(e),\ell(v_e^+)}$.
By~(\ref{Eqn:Yk_Cartier_Div}) and~(\ref{Eqn:Yk_pullback}),
the pullback $\varpi^{-1}(Y_{k,\fS})$ is thus a sub-locus of the right-hand side of~(\ref{Eqn:Yk_Cartier_Div}).
Moreover, it is a direct check that
$$
 \varpi^{-1}(Y_{k,\fE_i})\cap\fU=\{\ve_i\eq 0\}\qquad
\forall\,i\inn\bbI_+~\tn{with}~|\fE_i|\!\le\! k.
$$
Therefore, (\ref{Eqn:Yk_Cartier_Div}) holds.

Since every $\varpi^{-1}(Y_k)$  is a Cartier divisor of $\mtd{}$,
by the universality of the blowup $\pi\!:\ti\fM^\wt\!\lra\!\mwt$,
we obtain a unique morphism
$$
 \psi_1:\mtd{}\lra\ti\fM^\wt
$$
that $\varpi\!:\mtd{}\!\lra\!\mwt$ factors through. 

We next construct $\psi_2$ explicitly.
For any $\ti x\inn\ti\fM^\wt$,
let $(C,\bfw)$ be its image in $\mwt$.
As shown in~\cite[\S3.3]{HLN},
there exists a unique maximal sequence of exceptional divisors 
$$
 \ti E_{i_1},\ldots,\ti E_{i_k}\subset\tmwt,\qquad
 1\le i_1<\cdots<i_k
$$
containing $\ti x$.
Each $\ti E_{i_j}$ is obtained from blowing up along the proper transform of $Z_{i_j}$.
Note that $k$ is possibly 0, which means $(C,\bfw)$ is not in the blowup loci.
The weighted dual tree $\tau\eq(\ga_C,\bfw)$, along with the exceptional divisors $\ti E_{i_1},\ldots,\ti E_{i_k}$,
uniquely determines a weighted level tree $\lt_{\ti x}$ such that 
$$
 \bbI_+=\bbI_+(\lt_{\ti x})=\{-i_k,\ldots,-i_1\}.
$$
In particular, $\cht\eq\cht(\lt_{\ti x})\eq -i_k$.

With the line bundles $L_e$, $e\inn\Edg(\lt_{\ti x})$, as in~(\ref{Eqn:Le}), 
the notation $\lprp{\cdot,\cdot}_{\lt_{\ti x}}$ and $\lbrp{\cdot,\cdot}_{\lt_{\ti x}}$ as in~(\ref{Eqn:lrbr}),
and the notation $i[h]$ as in~(\ref{Eqn:i[h]}),
the line bundles
\begin{align}\label{Eqn:fL_i}
 \fL_i = L_{\se_i}\!\otimes
 \!\!\!\bigotimes_{j\in\lprp{\,i,\,i[1]\,}_{\lt_{\ti x}}}\!\!\!\!\fL_j^\vee
 \lra\mwt_{\tr},\qquad
 i\inn\bbI_+,
\end{align}
can be constructed inductively over $\bbI_+$.
Then, we take
\begin{align}\label{Eqn:fL_e}
 \fL_e = L_{e}\otimes\!\!\!\!
 \bigotimes_{j\in \lprp{\,\ell(e),\,\ell(v_e^+)\,}_{\lt_{\ti x}}}
 \!\!\!\!\!\!\!\!\!\fL_j^\vee\lra\mwt_\tr,\qquad
 e\inn\wh\Edg(\lt_{\ti x}).
\end{align}
In particular, $\fL_{\se_i}\eq\fL_i$.
For each $e\inn\Edg(\lt_{\ti x})$,
(\ref{Eqn:fL_i}) and (\ref{Eqn:fL_e}) imply
\begin{equation*}
\fL_e
 \!\otimes\!
 \bigotimes_{e'\succ e}\!
 (\fL_{e'}
 \!\otimes\!
 \fL_{\ell(e')}^\vee
 )
 =
 L_e^\succeq
 \otimes
 \big(\bigotimes_{\!\!
 \begin{subarray}{c}
 j\in\lprp{\,\ell(e),\,0\,}_{\lt_{\ti x}}
 \end{subarray}}\!\!\!\!\!\!\!\fL_j^\vee\,\big).
\end{equation*}
Hence for each $i\inn\bbI_+$, 
\begin{equation}\label{Eqn:blowup->tf}
 \mathring\P\Big(\bigoplus_{
 \begin{subarray}{c}
 \ell(v_e^-)=i
 \end{subarray}
 }\!\!\!\!
 \big(\fL_e
 \!\otimes\!
 \bigotimes_{e'\succ e}\!
 (\fL_{e'}
 \!\otimes\!
 \fL_{\ell(e')}^\vee
 )
 \big)
 \Big)
 \!=
 \mathring\P\big(\bigoplus_{
 \begin{subarray}{c}
 \ell(v_e^-)=i
 \end{subarray}
 }\!\!\!L_e^\succeq\,\big).
\end{equation} 

For  $h\!\ge\!1$,
let $\ti x_{(h)}$ be the image of $\ti x$ in the exceptional divisor of the $h$-th step.
Given $i\inn\bbI_+$,
The proper transform of $Z_{-i}$ after the first $-i\!-\!1$ steps of the blowup may have several connected components;
see~\cite[Lemma 3.3.2]{HLN}.
The normal bundle of the component containing $\ti x_{(-i-1)}$ is the pullback
$\pi_{(-i-1)}^*\bigoplus_{e\in\fE_{i}}\!\!\fL_e$,
where $\pi_{(h)}\!:\ti\fM^\wt_{(h)}\!\lra\!\mwt$ is the blowup after the $h$-th step.

Notice that the non-zero entries of $\ti x_{(-i)}$ exactly correspond to the edges $e\inn\fE_{i}$ satisfying $\ell(v_e^-)\eq i$.
Therefore,
$$
 \ti x_{(-i)}\,\in\,
 \pi^*_{(-i-1)}\mathring\P\big(\!\bigoplus_{
 \begin{subarray}{c}
 \ell(v_e^-)=i
 \end{subarray}
 }\!\!\!\fL_e\,\big).
$$ 
Then, $\ti x_{(-j)}$ with $j\inn\lbrp{i,0}_{\lt_{\ti x}}$ together  determine a unique 
\begin{equation*}\label{Eqn:blowup_intermediate}
 \eta_i(\ti x)\in 
 \mathring\P\Big(\bigoplus_{
 \begin{subarray}{c}
 \ell(v_e^-)=i
 \end{subarray}
 }\!\!\!\!
 \big(\fL_e
 \!\otimes\!
 \bigotimes_{e'\succ e}\!
 (\fL_{e'}
 \!\otimes\!
 \fL_{\ell(e')}^\vee
 )
 \big)
 \Big)
 =
 \mathring\P\big(\bigoplus_{
 \begin{subarray}{c}
 \ell(v_e^-)=i
 \end{subarray}
 }\!\!\!L_e^\succeq\,\big).
\end{equation*} 
The last equality above follows from~(\ref{Eqn:blowup->tf}).
We then set
\begin{equation*}
 \psi_2(\ti x)=
 \Big((C,\bfw),[\lt_{\ti x}],
 \big(\eta_i(\ti x):i\inn\bbI_+(\lt_{\ti x})\big)
 \Big)\quad
 \in\mtd{[\lt_{\ti x}]}.
\end{equation*}
Obviously, this implies $\varpi\!\circ\!\psi_2\eq\pi$.

It remains to verify such defined $\psi_2$ is a morphism.
Let $\cV\!\lra\!\mwt$ be a smooth chart containing $(C,\bfw)$,
and $\{\ze_e\!:e\inn\Edg(\lt_{\ti x})\}\!\sqcup\!\{\vs_j\!:j\inn J\}$ be a set of local parameters centered at $(C,\bfw)$ as in~(\ref{Eqn:local_parameters}).
As shown in~\cite[\S3.1\&\S3.3]{HLN},
there exists a chart $\ti\cV_{\ti x}\!\lra\!\ti\fM^\wt$ containing $\ti x$ with local parameters
\begin{gather*}
 \ti\ve_i,~i\inn\bbI_+;\qquad
 \rho_e,~e\inn\wh\Edg(\lt_{\ti x})\bsl\big(\bbI_\cht\!\sqcup\!\{\se_i\!:i\inn\bbI_+\}\big);\\
 \wc z_e,~e\inn\bbI_\cht;\qquad
 \ti z_e,~e\inn\bbI_-;\qquad
 s_j,~j\inn J.
\end{gather*}
All $\rho_e$ are nowhere vanishing on $\ti\cV_{\ti x}$.
Moreover,
with
$
 \pi\!:\ti\cV_{\ti x}\!\lra\!\cV
$
denoting the blowup,
we have 
\begin{gather*}
 \pi^*\ze_{\se_i}\!\eq\!\!\!\!\!\prod_{h\in\lbrp{i,i[1]}_{\lt_{\ti x}}}\!\!\!\!\!\!\!
 \ti\ve_h\ \ \forall\,i\inn\bbI_+;
 \quad
 \pi^*\ze_e\eq\rho_e\!\!\!\!\!\!\!\!\prod_{i\in\lbrp{\ell(e),\ell(v_e^+)}_{\lt_{\ti x}}}\!\!\!\!\!\!\!\!\!\!\!\ti\ve_i\ \ \;\forall\,e\inn\wh\Edg(\lt_{\ti x})\bsl\big(\bbI_\cht\!\sqcup\!\{\se_i\}_{i\in\bbI_+}\big);
 \\
 \pi^*\ze_e\eq\wc z_e\!\!\!\!\!\!\!\!\prod_{i\in\lbrp{\ell(e),\ell(v_e^+)}_{\lt_{\ti x}}}\!\!\!\!\!\!\!\!\!\!\!\ti\ve_i\ \ \forall\,e\inn\bbI_\cht;\qquad
 \pi^*\ze_e\eq\ti z_e\ \ \forall\,e\inn\bbI_-;\qquad
 \pi^*\vs_j\eq s_j\ \ \forall\,j\inn J.
\end{gather*}
For $e\inn\{\se_i\}_{i\in\bbI_+}$, we set $\rho_e\eq 1$. Then,
$$
 \rho_e
 \in\Ga\big(\sO_{\ti\cV_{\ti x}}^*\big)\qquad
 \forall\  
 e\inn\wh\Edg(\lt_{\ti x})\bsl\bbI_\cht.
$$

Let $\fU_{\psi_2(\ti x)}\!\lra\!\mtd{}$ be a twisted chart centered at $\psi_2(\ti x)$, lying over $\cV\!\lra\!\mwt$.
The parameters on $\fU_{\psi_2(\ti x)}$ are as in~(\ref{Eqn:local_parameters}).
It is a direct check that the point-wise defined $\psi_2$ can locally be written as
$$
 \psi_2:\ti\cV_{\ti x}\lra \fU_{\psi_2(\ti x)}
$$
such that
\begin{gather*}
 \psi_2^*\ve_i\eq \ti\ve_i\ \ \forall\,i\inn\bbI_+;\qquad
 \psi_2^*u_e\eq
 \frac{\prod_{\fe\succeq e}\rho_{\fe}}
 {\prod_{\fe\succeq \se_{\ell(e)}}\!\rho_{\fe}}\ \ \forall\,e\inn\wh\Edg(\lt)\bsl\big(\bbI_\cht\!\sqcup\!\{\se_i\}_{i\in\bbI_+}\!\big);
  \\ 
 \psi_2^*u_e\eq\wc z_e\!\cdot\!\frac{\prod_{\fe\succ e}\rho_{\fe}}
 {\prod_{\fe\succeq \se_{\ell(e)}}\!\rho_{\fe}}\ \ \forall\,e\inn\bbI_\cht;\quad
 \psi_2^*z_e\eq\ti z_e\ \ \forall\,e\inn\bbI_-;\quad
 \psi_2^*w_j\eq s_j\ \ \forall\,j\inn J.
\end{gather*}
This shows $\psi_2\!:\ti\fM^\wt\!\lra\!\mtd{}$ is a morphism.
\end{proof}

\begin{rmk}
In~\cite{HL11},
another resolution $\ti\fM^{\tn{dr}}\!\lra\!\mwt$,
called the \ts{derived resolution} of $\mwt$,
is constructed for the purpose of diagonalizing certain direct image sheaves.
That resolution is ``smaller'' in that the resolution $\tmwt\!\lra\!\mwt$  of~\cite{HL10} factors through $\ti\fM^{\tn{dr}}\!\lra\!\mwt$.
Mimicking the approach of \S3,
we may construct a moduli stack
$$
 \fN=\bigsqcup_{[\lt]\in\sT_\nL^\wt}\!\!\!\fN_{[\lt]},\qquad
 \fN_{[\lt]}=
 \mathring\P\Big(\!
   \bigoplus_{
   \begin{subarray}{c}
   e\in\Edg(\lt),\,
   \ell(v_e^-)=\cht(\lt)
   \end{subarray}
   }\!\!\!\!\!\!\!\!\!\!\!\!\!\!
   L_e^\succeq\ \ \Big)
  \lra\mwt_{\ff[\lt]}.
$$ 
This moduli should be isomorphic to $\ti\fM^{\tn{dr}}$. \\
\end{rmk}

\end{document}